\documentclass[11pt,reqno,oneside]{amsart}
\usepackage[utf8]{inputenc}
\usepackage[english]{babel}
\usepackage[babel]{csquotes}
\usepackage[backend=bibtex, hyperref, maxnames=4]{biblatex}
\bibliography{CurieWeissType}

\usepackage{amsmath}
\usepackage{amsfonts}
\usepackage{amssymb}
\usepackage{amsthm} 
\usepackage{dsfont} 
\usepackage[mathscr]{euscript} 
\usepackage{eurosym}
\usepackage{bbm} 
\usepackage[a4paper]{geometry}
\usepackage{comment} 
\usepackage{paralist} 

\usepackage{graphicx}
\usepackage{xkeyval}
\usepackage{pstricks}
\usepackage{hyperref}
\usepackage{longtable}
\usepackage{tikz}

\hypersetup{
urlcolor=black, 
  menucolor=black, 
  citecolor=black, 
  anchorcolor=black, 
  filecolor=black, 
  linkcolor=black, 
  colorlinks=true,
}
\usepackage{comment}

\usetikzlibrary{matrix}

\newcommand{\lebesgue}{\lambda\mspace{-7mu}\lambda}

\newcommand{\Prob}{\mathds{P}}

\newcommand{\E}{\mathds{E}}

\newcommand{\C}{\mathbb{C}}
\newcommand{\R}{\mathbb{R}}

\newcommand{\Z}{\mathbb{Z}}
\newcommand{\N}{\mathbb{N}}

\newcommand{\mc}[1]{\mathcal{#1}}
\newcommand{\abs}[1]{|{#1}|} 
\newcommand{\bigabs}[1]{\left|{#1}\right|} 
\newcommand{\one}{\mathds{1}}
\newcommand{\Mcal}{\mathcal{M}}
\newcommand{\Dcal}{\mathcal{D}}
\newcommand{\Acal}{\mathcal{A}}
\newcommand{\Scal}{\mathcal{S}}
\newcommand{\Bcal}{\mathcal{B}}
\newcommand{\Ccal}{\mathcal{C}}
\newcommand{\Fcal}{\mathcal{F}}

\newcommand{\Gcal}{\mathcal{G}}

\newcommand{\Ical}{\mathcal{I}}

\newcommand{\Ecal}{\mathcal{E}}
\newcommand{\Wcal}{\mathcal{W}}

\newcommand{\Tcal}{\mathcal{T}}

\newcommand*{\defeq}{\mathrel{\vcenter{\baselineskip0.5ex \lineskiplimit0pt
                     \hbox{\scriptsize.}\hbox{\scriptsize.}}}%
                     =}

\newcommand{\de}{\text{d}}

\newcommand{\norm}[1]{\|#1\|}
\newcommand{\bignorm}[1]{\left\|#1\right\|}

\newcommand{\oneto}[1]{[{#1}]}

\renewcommand{\Re}{\operatorname{Re}}
\renewcommand{\Im}{\operatorname{Im}}

\DeclareMathOperator{\tr}{tr}

\theoremstyle{plain}
\newtheorem{lemma}{Lemma}
\newtheorem{theorem}[lemma]{Theorem}

\newtheorem{corollary}[lemma]{Corollary}
\theoremstyle{definition}
\newtheorem{definition}[lemma]{Definition}

\newtheorem{example}[lemma]{Example}
\newtheorem{remark}[lemma]{Remark}
\theoremstyle{remark}

\newcommand{\map}[5]{
		\begin{align}
		{#1}: {#2}\ &\longrightarrow \ {#3}\notag\\
		      {#4}\ &\longmapsto \ \notag {#5}
		\end{align}
		}

\begin{document}
\title[Curie-Weiss Type Ensembles]
{Local Semicircle Law for Curie-Weiss Type Ensembles}
\author[Michael Fleermann]{Michael Fleermann}
\author[Werner Kirsch]{Werner Kirsch}
\author[Thomas Kriecherbauer]{Thomas Kriecherbauer}
\begin{abstract}
We derive local semicircle laws for random matrices with exchangeable entries which exhibit correlations that decay at a very slow rate. In fact, any $\ell$-point correlation $\E [Y_1\cdots Y_{\ell}]$ between distinct matrix entries $Y_1,\ldots,Y_{\ell}$ may decay at a rate of only $N^{-\ell/2}$. We call our ensembles \emph{of Curie-Weiss type}, and Curie-Weiss($\beta$)-distributed entries directly fit within our framework as long as $\beta\in [0,1]$. Using rank-one perturbations, we show that even in the high-correlation regime $\beta\in(1,\infty)$, where $\ell$-point correlations survive in the limit, the local semicircle law still holds after rescaling the matrix entries with a constant which depends on $\beta$ but not on $N$. 

\end{abstract}
\keywords{random matrix, local semicircle law, exchangeable entries, correlated entries, Curie-Weiss entries}
\subjclass[2010]{60B20.} 
\maketitle

\section{Introduction}

The local semicircle law is a relatively recent result that was derived to gain a more detailed understanding of the convergence of the empirical spectral distributions (ESDs) of random matrices to the semicircle distribution. Further, it was also used to establish universality results for Wigner matrices. A common formulation of this type of theorem is a uniform alignment of the Stieltjes transforms of the ESDs $\sigma_N$ and the semicircle distribution $\sigma$, see \autocite{AnttiLLSurvey}, for example. Another formulation of the local law is as follows, cf.\ \autocite{Tao:Vu:2012}: For any sequence of intervals $(I_N)_N$, whose diameter is not decaying to zero too quickly, $\sigma_N(I_N)$ can be well approximated by $\sigma(I_N)$ for large $N$. In fact, the second formulation of the local law will follow from the first, as we will show further below in Theorem~\ref{thm:TaoVuWLL}. 
And it is precisely this second formulation which lends the local law its name: Even when zooming in onto smaller and smaller intervals, the ESDs are well-approximated by the semicircle distribution (see also \autocite{fleermann:kirsch:toth:2020b} for a translation of this convergence concept to the setting of classical probability theory). 

Although there were some previous results into the direction of a local law in \autocite{Khorunzhy1997} and \autocite{Erdos2009a}, it is safe to say that on the level of strength available today, it was established by Erd{\H{o}}s, Schlein and Yau in \autocite{Erdos2009b} and by Tao and Vu in \autocite{Tao:Vu:2011}. Ever since, the results were strengthened (see \autocite{GotzeNaumov2016} and \autocite{2019_Gotze_LLFourthMoment}, for example) and proof layouts were refined to make the theory more accessible to a broader audience. Indeed, the local laws are displayed in pedagogical manner in the text \autocite{AnttiLLSurvey} by Benaych-Georges and Knowles and the book \autocite{Erdos} by Erdős and Yau. Both of these texts have their roots in the joint publication \autocite{ErdosAntti2013b}. 

 As the semicircle law itself, the local semicircle law was initially considered for random matrices with independent and identically distributed entries, see \autocite{Erdos2009b}. Further generalizations can be found in \autocite{ErdosAntti2013b}, where entries are still assumed to be independent, but not identically distributed anymore.
 
  Of course, the next question is if and how global and local laws can also be proved for random matrices with correlated entries. With respect to local laws, the following results for ensembles with correlations can be found in the literature: In \autocite{AjankiErdos2016}, the local law was proved for random matrices with correlated Gaussian entries, where the covariance matrix is assumed to possess a certain translation invariant structure. In \autocite{AjankiErdos2018}, ensembles with correlated entries were considered, where the correlation decays arbitrarily polynomially fast in the distance of the entries. This result has been improved by \autocite{ErdosKruger2017} (who reference an older preprint version of \autocite{AjankiErdos2018}), where fast polynomial decay is assumed only for entries outside of neighborhoods of a size growing slower than $\sqrt{N}$, and a slower correlation decay between entries within these neighborhoods. Another correlation structure was analyzed in \autocite{ZiliangChe2017}, where correlation was only allowed for entries close to each other and independence was assumed otherwise.  What all four mentioned publications have in common is that the local \emph{semicircle} law is not the main object of interest, but rather the existence of \emph{some} local limit.
  
   In this paper, we will derive various forms of local semicircle laws for a random matrix ensemble with slow correlation decay for the entries and with no additional assumptions on the spatial structure. In fact, the $\ell$-point correlation $\E[Y_1\cdots Y_{\ell}]$ between any $\ell$ distinct matrix entries in the upper right half of the matrix may decay at a rate of order $N^{-\ell/2}$ (see Definition~\ref{def:CWtype} for more details). In particular, our model is not covered by the previous work on correlated entries that was mentioned above (for example, in \autocite{ErdosKruger2017}, Assumption (D) requires a faster decay), and new techniques of proof must be developed, namely new sets of so called large-deviation inequalities, see Theorems~\ref{thm:largedev} and~\ref{thm:largedev2}. The ensemble we study will be called "of Curie-Weiss type", and not surprisingly, Curie-Weiss($\beta$)-distributed entries will be directly admissible to our framework as long as $\beta\leq 1$, and indirectly admissible via a rank-one perturbation when $\beta>1$. It should be noted that at the critical temperature $\beta=1$, the $l$-point correlation mentioned above will decay \emph{exactly} at the rate of $N^{-l/2}$, so our condition is tight for a relevant example.
   
   The Curie-Weiss($\beta$) distribution on the space of spin configurations $\{\pm 1\}^N$ is used to model ferro-magnetic behavior. Here, $\beta>0$ is the \emph{inverse temperature}, a model parameter with great influence on the asymptotic properties of the spins. Global laws for random matrices with Curie-Weiss spins have so far been investigated in \autocite{Friesen:zwei}, where independent diagonals were filled with Curie-Weiss entries, in \autocite{HKW}, where the full upper right triangles were Curie-Weiss distributed, in \autocite{Kirsch:Kriecherbauer:2018}, where the temperature was allowed to drop to sub-critical levels, in \autocite{Fleermann:Kirsch:Kriecherbauer:2021}, where band matrices with Curie-Weiss spins were investigated, and previous semicircle laws in \autocite{HKW} and \autocite{Kirsch:Kriecherbauer:2018} were strengthened to hold almost surely, and in \autocite{Fleermann:Heiny:2020}, where limit laws for sample covariance matrices with Curie-Weiss entries were derived.

This work continues the analysis of the first author in \autocite{FleermannDiss}, where he answered a question of the second author. The main motivation was to derive local laws ensembles with Curie-Weiss distributed entries, see Example~\ref{ex:CWensemble}. What is crucial in our analysis is that these ensembles are of de-Finetti type, see Definition~\ref{def:deFinetti}. In Definition~\ref{def:CWtype} we formulate sufficient conditions for such ensembles that allow to prove a weak form of the local law, Theorem~\ref{thm:WLL}. Somewhat surprisingly, this local law also holds in the case of sub-critical temperatures, where correlations do not decay at all, Theorem~\ref{thm:betalarger1}. This result requires the construction of a suitable probability space and an auxiliary matrix ensemble that allows to make use of Theorem~\ref{thm:WLL}, together with a finite-rank perturbation argument. Since the ensembles of Definition~\ref{def:CWtype} allow to treat Curie-Weiss ensembles (but are not restricted to them), we call them \emph{ensembles of Curie-Weiss type}.
The proof of Theorem~\ref{thm:WLL} follows the strategy presented in \autocite{AnttiLLSurvey} for the case of independent entries, and a number of their results can be used. In Section 3 we present the novel arguments that we need to treat the ensembles of Curie-Weiss type. 

In the Appendix, we present various extensions and corollaries of Theorem~\ref{thm:WLL}. With help of the general Lemma~\ref{lem:supinside}, we extend the uniformness of Theorem~\ref{thm:WLL} in Theorem~\ref{thm:simulWLL}. We use this result in combination with Lemma~\ref{lem:stieltjestosemicircle} to prove Theorem~\ref{thm:uniformkernelconvergence}, which analyzes the approximation of the semicircle density by a kernel density estimate which is based on the empirical spectral distribution. Lastly, in Theorem~\ref{thm:smallscales} and Theorem~\ref{thm:TaoVuWLL} we analyze absolute and relative differences of interval probabilities of the empirical spectral distributions and the semicircle distribution.

\section{Setup and Main Results}

\subsection{Ensembles of Curie-Weiss type}

We will first explain some notation and introduce random matrices of Curie-Weiss type. The expectation operator $\E$ will always denote the expectation with respect to a generic probability space $(\Omega,\Acal,\Prob)$. Euclidian spaces $\R^n$ will always be equipped with Borel-$\sigma$-algebras induced by the standard topology. The space $\Mcal_1(\R)$ of all probability measures on $\R$ will be equipped with the topology of weak convergence and the associated Borel $\sigma$-algebra. In addition, probability spaces with finite sample space will always be equipped with the power set as $\sigma$-algebra. If $I$ is an index set and for all $i\in I$, $Z_i$ is a mathematical object, then we write $Z_I \defeq (Z_i)_{i\in I}$. On the other hand, if for all $i\in I$, $M_i$ is a set, then we write $M^I \defeq \prod_{i\in I} M_i$ as the cartesian product. Lastly, if we write $a=a(b)$, where $a$ is an expression and $b$ is a parameter vector, then this means that $a$ depends on the choice of $b$. The following definition is based on \autocite{KirschMomentSurvey} and \autocite{Kirsch:Kriecherbauer:2018}.

\begin{definition}
\label{def:deFinetti}
Let $I$ be a finite index set and $Y_I$ be a family of $\R$-valued random variables on some probability space $(\Omega,\Acal,\Prob)$. Then the random vector $Y_I$ is called \emph{of de-Finetti type}, if there is a probability space $(T,\Tcal,\mu)$ and a measurable mapping 
\map{P}{(T,\Tcal)}{\Mcal_1(\R)}{t}{P_t}
such that for all measurable sets $B\subseteq \R^I$, we find
\begin{equation}
\label{eq:deFinettiEquation}
\Prob(Y_I \in B) = \int_{T} P^{\otimes I}_t (B) \de\mu(t),
\end{equation}
where $P^{\otimes I}_t \defeq \otimes_{i\in I} P_t$ is the $I$-fold product measure on $\R^I$. 
\end{definition}

It should be noted that for a random vector $Y_I$ to be of de-Finetti type is solely a property of the \emph{distribution} of $Y_I$ and not a property of the probability space on which $Y_I$ is defined. To be more precise, it means that the push-forward distribution $\Prob^{Y_I}$ is a mixture of product distributions $P^{\otimes I}_t$, $t\in T$. In this context, $\mu$ is also called \emph{mixing distribution} or simply \emph{mixture}. We will also call $(T,\Tcal,\mu,P)$ \emph{mixing space}.
Further properties of de-Finetti type variables are clarified in the following remark:

\begin{remark}
Let $Y_I$ be of de-Finetti type as in Definition~\ref{def:deFinetti}, then we observe:
\begin{enumerate}
	\item For any subset $J\subseteq I$, $Y_J$ is of de-Finetti type with respect to the same mixing space $(T,\Tcal,\mu,P)$.
	\item For any $t\in T$, the coordinates of the identity map on $(\R^I, P^{\otimes I}_t)$ are i.i.d.\ $P_t$-distributed.
	\item  The random variables $Y_I$ are  \emph{exchangeable}, that is, if $\pi:I\to I$ is a bijection, then $(Y_i)_{i\in I}$ and $(Y_{\pi(i)})_{i\in I}$ have the same distribution.
 \end{enumerate}
 \end{remark}
 
 \begin{lemma}
\label{lem:expdefinetti}
Let $Y_I$ be of de-Finetti type with respect to the mixing space $(T,\mc{T},\mu,P)$. Then it holds for any measurable function $F: \R^I\to \C$:
\[
\E F(Y_I) = \int_{T} \int_{\R^I} F(y_I) \de P^{\otimes I}_t(y_I)\de\mu(t),
\]	
\end{lemma}
where the left-hand side of the equation is well-defined iff the right-hand side is.
\begin{proof}
The statement follows by standard arguments: The claim is easily verified for step functions of the form $F = \sum_{k=1}^K\alpha_k\one_{A_k}$, where $K\in\N$, $\alpha_k\geq 0$ and $A_k\subseteq\R^k$ are measurable. The case for $F\geq 0$ is then concluded via Beppo-Levi. The $\R$-valued case is seen by decomposing $F= F_+ - F_-$, and the final $\C$-valued case is then shown by decomposing $F=\Re F + i\Im F$.
\end{proof}

A prominent example of random variables of de-Finetti type is given by Curie-Weiss spins:

\begin{definition}\label{def:curieweiss}
Let $N\in\N$ be arbitrary and $Y_1,\ldots,Y_N$ be random variables defined on some probability space $(\Omega,\Acal,\Prob)$. Let $\beta\geq 0$, then we say that $Y_1,\ldots,Y_N$ are Curie-Weiss($\beta$,$N$)-distributed, if for all $y_1,\ldots,y_N\in\{-1,1\}$ we have that
\[
\Prob(Y_1=y_1,\ldots,Y_N=y_N) = \frac{1}{Z_{\beta,N}}\cdot e^{\frac{\beta}{2N}\left(\sum y_i\right)^2},
\]
where $Z_{\beta,N}$ is a normalization constant. The parameter $\beta$ is called \emph{inverse temperature}.
\end{definition} 

The Curie-Weiss($\beta,N$) distribution is used to model the behavior of $N$ ferromagnetic particles (spins) at the inverse temperature $\beta$. At low temperatures, that is, if $\beta$ is large, all magnetic spins are likely to have the same alignment, resembling a strong magnetic effect. In contrast, at high temperatures (if $\beta$ is small), spins can act almost independently, resembling a weak magnetic effect. At infinitely high temperature, that is, if $\beta=0$, Curie-Weiss spins are simply i.i.d.\ Rademacher distributed random variables. For details on the Curie-Weiss model we refer to \autocite{Ellis}, \autocite{Thompson} and \autocite{KirschMomentSurvey}. The Curie-Weiss distribution is an important model in statistical mechanics. It is exactly solvable and features a phase transition at $\beta=1$. The behavior of Curie-Weiss spins differs significantly in the regimes $\beta=0$, $\beta\in(0,1)$, $\beta=1$ and $\beta\in(0,\infty)$, as exemplified by the next lemma. In particular, we will see exactly at which speed $\ell$-point correlations between Curie-Weiss spins decay, and that for $\beta>1$ these correlations do not vanish at all:

\begin{lemma}
\label{lem:CWcorrelations}
Fix $\ell\in \N$ and let for all  $N\geq \ell$,  $(Y^{(N)}_1,\ldots,Y^{(N)}_{\ell})$ be part of a Curie-Weiss($\beta,N$) distributed random vector.  If $\ell$ is even, the following statements hold:
\begin{enumerate}[i)]
\item If $\beta=0$, then $\E Y^{(N)}_1\cdots Y^{(N)}_{\ell} = 0$. 
\item If $\beta\in(0,1)$, then for some constant $c=c(\beta,\ell)>0$: 
\[
\E Y^{(N)}_1\cdots Y^{(N)}_{\ell} \sim c N^{-\ell/2} \text{ as } N\to\infty.
\]
\item If $\beta=1$, then for some constant $c=c(\ell)>0$:
\[
\E Y^{(N)}_1\cdots Y^{(N)}_{\ell} \sim c N^{-\ell/4}\text{ as }N\to\infty.
\]
\item If $\beta\in(1,\infty)$, then 
\[
\E Y^{(N)}_1\cdots Y^{(N)}_{\ell} \sim c^{\ell}
\]
as $N\to\infty$, where $c=c(\beta)\in (0,1)$ is the unique positive number such that $\tanh(\beta c)=c$.
\end{enumerate}
If $\ell$ is odd, then for all $\beta\geq 0$ one has $\E Y^{(N)}_1\cdots Y^{(N)}_{\ell} = 0$.
\end{lemma}
\begin{proof}
See Theorem 5.17 in \cite{KirschMomentSurvey}.
\end{proof}

The next theorem shows that the discrete distribution of Curie-Weiss spins has a de-Finetti representation in the sense of Definition~\ref{def:deFinetti}.

\begin{theorem}
\label{thm:curiedefinetti}
If $Y_1,\ldots,Y_N$ are Curie-Weiss($\beta$,$N$)-distributed with $\beta\geq 0$, then they are of de-Finetti type with respect to the mixing space $((-1,1),\Bcal_{(-1,1)},\mu^{\beta}_N,P)$, where
\map{P}{(-1,1)}{\mc{M}_1(\R)}{t}{P_t = \frac{1-t}{2}\delta_{-1} + \frac{1+t}{2}\delta_1.} Here, $\Bcal_{(-1,1)}$ denotes the Borel $\sigma$-algebra over the interval $(-1,1)$ and $\mu^{\beta}_N$ is the Dirac measure $\delta_0$ for $\beta=0$, whereas if $\beta>0$, $\mu^{\beta}_N$ is the Lebesgue-continuous probability distribution with density on $(-1,1)$ given by
\[
t\mapsto f_N(t) \defeq C\cdot \frac{e^{-\frac{N}{2} F_{\beta}(t)}}{1-t^2}\one_{(-1,1)}(t),
\]
where $C=C(\beta,N)$ is a normalization constant and for all $t\in(-1,1)$ we define
\[
F_{\beta}(t) \defeq \frac{1}{\beta} \left(\frac{1}{2}\ln\left(\frac{1+t}{1-t}\right)\right)^2 + \ln(1-t^2).
\]	
Further, if $\beta\leq 1$, the mixtures $(\mu^{\beta}_N)_{N\in\N}$ satisfy the following moment decay:
\[
\forall~p\in 2\N: \int_{(-1,+1)} t^p \de\mu^{\beta}_N(t) \leq \frac{K_{\beta,p}}{N^{\frac{p}{4}}},   
\]
where $K_{\beta,p}\in \R_+$\label{sym:CWdefinetticonstant} is a constant that depends on $\beta$ and $p$ only. 
\end{theorem}
\begin{proof}
This was shown rigorously in \autocite{KirschMomentSurvey}, see Theorem 5.6, Remark 5.7, Proposition 5.9 and Theorem 5.17 in their text. 
\end{proof}

The Curie-Weiss type ensembles (sequences of random matrices) which we study in this paper are defined as follows:

\begin{definition}
\label{def:CWtype}
An ensemble of real symmetric random matrices $N\times N$ matrices $(H_N)_N$ is called \emph{of Curie-Weiss type}, if:
\begin{enumerate}
\item[a)] For all $N\in\N$ it holds
\[
\left(H_N(i,j)\right)_{1\leq i\leq j\leq N}=\left(\frac{1}{\sqrt{N}}X_N(i,j)\right)_{1\leq i\leq j\leq N},
\] 
where $(X_N(i,j))_{1\leq i\leq j\leq N}$ is of de-Finetti type with respect to some mixing space $(T_N, \mc{T}_N,\mu_N,P^{(N)})$.
\item[b)] Set for all $\ell,N\in\N$ and $t\in T_N$, $m^{(\ell)}_N(t)\defeq \int_{\R} x^{\ell} \de P^{(N)}_t(x)$ the $\ell$-th moment of $P^{(N)}_t$. Then it holds:
\begin{align}
&\forall\,p\in 2\N:\,\exists\,K_p\in\R_+:\,\forall\,N\in\N:\notag\\
&\text{First moment condition:}  \int_{T_N} \abs{m^{(1)}_N(t)}^p \de\mu_N(t) \leq \frac{K_p}{N^{\frac{p}{2}}}\label{eq:firstmoment}\\
&\text{Second moment condition:} \int_{T_N} \abs{1 - m^{(2)}_N(t)}^p \de\mu_N(t) \leq \frac{K_p}{N^{\frac{p}{2}}}\label{eq:secondmoment}\\
&\text{Central first moment condition:} \sup_{t\in T_N}\int_{\R}\abs{y-m^{(1)}_N(t)}^p\de P^{(N)}_t(y) \leq K_p\label{eq:centralfirstmoment}\\
&\text{Central second moment condition:} \sup_{t\in T_N}\int_{\R}\abs{y^2-m^{(2)}_N(t)}^p\de P^{(N)}_t(y) \leq K_p\label{eq:centralsecondmoment}
\end{align}

\end{enumerate}
\end{definition}

Notationally, for the remainder of this paper, we set $\oneto{N}\defeq\{1,\ldots,N\}$ for all $N\in\N$.
\begin{example}
\label{ex:CWensemble}
Let $0\leq  \beta$ be arbitrary and let for each $N\in\N$ the random variables $(\tilde{X}_N(i,j))_{i,j\in\oneto{N}}$ be Curie-Weiss($\beta,N^2$)-distributed. Define the \emph{Curie-Weiss($\beta$) ensemble} $(H_N)_N$ by setting
\[
\forall\,N\in\N:\,\forall\, (i,j)\in\oneto{N}^2:~H_N(i,j) =
\begin{cases}
\frac{1}{\sqrt{N}}\tilde{X}_N(i,j) & \text{if $i\leq j$}\\
\frac{1}{\sqrt{N}}\tilde{X}_N(j,i) & \text{if $i> j$}.
\end{cases}.
\]
If $\beta\in[0,1]$, by Theorem~\ref{thm:curiedefinetti}, $(H_N)_N$ is an ensemble of Curie-Weiss type with mixtures $(\mu_N)_N\defeq (\mu^{\beta}_{N^2})_N$. To see this, condition a) in Definition~\ref{def:CWtype} is clear by construction. Note that the space $(T_N,\mc{T}_N)$ and the map $P^{(N)}$ are the same for all $N$, only the mixture $\mu_N$ changes with $N$. For condition b), note that $m^{(1)}_N(t)=t$ and $m^{(2)}_N(t)=1$ for all $N\in\N$ and $t\in(-1,1)$. So by Lemmas~\ref{lem:expdefinetti} and~\ref{lem:CWcorrelations},  Conditions \eqref{eq:firstmoment}, \eqref{eq:secondmoment}, \eqref{eq:centralfirstmoment} and \eqref{eq:centralsecondmoment} are satisfied.
\end{example}

\subsection{Stochastic Domination, Resolvents and Stieltjes transforms}

For the statement of the local law and its proof we need the concepts of stochastic domination, resolvents and Stieltjes transforms. The first time the concept of stochastic domination was used was in \autocite{ErdosAntti2013a}. We will say that a statement which depends on $N\in\N$ holds \emph{$v$-finally}, where $v$ is a parameter(-vector), if the statement holds for all $N\geq N^*(v)$.

\begin{definition}
\label{def:stochdom}
Let $X=X^{(N)}$ be a sequence of complex-valued and $Y=Y^{(N)}$ be a sequence of non-negative random variables, then we say that $X$ is \emph{stochastically dominated} by $Y$, if for all $\epsilon,D>0$ there is a constant $C_{\epsilon,D}\geq 0$\label{sym:precconstants} such that
\[
\forall\ N\in\N: \Prob\left(\abs{X^{(N)}} > N^{\epsilon}Y^{(N)}\right) \leq \frac{C_{\epsilon,D}}{N^D}.
\]
In this case, we write $
X\prec Y$ or $X^{(N)}\prec Y^{(N)}$.
If both $X$ and $Y$ depend on a possibly $N$-dependent index set $U=U^{(N)}$, such that $X = \left(X^{(N)}(u), N\in\N, u\in U^{(N)}\right)$ and $Y = \left(Y^{(N)}(u), N\in\N, u\in U^{(N)}\right)$, then we say that $X$ is stochastically dominated by $Y$ \emph{uniformly in} $u\in U$, if for all $\epsilon, D>0$ we can find a $C_{\epsilon,D}\geq0$ such that
\begin{equation}
\label{eq:stochdom}	
\forall\, N\in\N: \sup_{u\in U^{(N)}} \Prob\left(\abs{X^{(N)}(u)} > N^{\epsilon} Y^{(N)}(u)\right) \leq \frac{C_{\epsilon,D}}{N^D}.
\end{equation}
In this case, we write $X\prec Y$ or  $X(u)\prec Y(u), u\in U$ or $X^{(N)}(u)\prec Y^{(N)}(u), u\in U^{(N)}$, where the first version is used if $U$ is clear from the context. In above situation, if all $Y(u)$ are strictly positive, then we say that $X$ is stochastically dominated by $Y$, \emph{simultaneously in} $u\in U$, if for all $\epsilon, D>0$ we can find a $C_{\epsilon,D}\geq 0$, such that
\[
\forall\, N\in\N: \Prob\left(\sup_{u\in U^{(N)}} \frac{\abs{X^{(N)}(u)}}{Y^{(N)}(u)} > N^{\epsilon} \right) \leq \frac{C_{\epsilon,D}}{N^D},
\]
and then we write $
\sup_{u\in U} \abs{X(u)}/Y(u) \prec 1$ or $\sup_{u\in U^{(N)}} \abs{X^{(N)}(u)}/Y^{(N)}(u)\prec 1$. 
\end{definition}

\begin{remark}
\label{rem:suffices}
Simultaneous stochastic domination implies uniform stochastic domination (for the other direction, see Lemma~\ref{lem:supinside}). Further,
 in order to show $X\prec Y$, it suffices to show that \eqref{eq:stochdom} holds for all $\epsilon$ small enough, that is, for all $\epsilon\in(0,\epsilon_0]$ for some $\epsilon_0>0$. In addition, it suffices to show \eqref{eq:stochdom} for $(\epsilon,D)$-finally all $N\in\N$.
 \end{remark}

Stochastic domination admits several important and intuitive rules of calculation. For example, $\prec$ is transitive and reflexive, and if $X_1\prec Y_1$ and $X_2\prec Y_1$, then both $X_1 + X_2 \prec Y_1 + Y_2$ and $X_1\cdot X_2 \prec Y_1\cdot Y_2$. For more rules of calculation and their proofs, see e.g.\ \autocite{FleermannDiss}. In what follows, we will follow largely the notation in \autocite{AnttiLLSurvey}. In particular, we will drop the index $N$ from many -- but not all -- $N$-dependent quantities. Let $H=H_N$ be an ensemble of Curie-Weiss type, $z\in\C\backslash\R$, then we denote by $G(z)\defeq (H-z)^{-1}$ its resolvent at $z$. The resolvent $G$ of $H$ carries all the spectral information of $H$ which is contained in its empirical spectral distribution 
\begin{equation}
\label{eq:ESD}	
\sigma = \sigma_N \defeq \frac{1}{N}\sum_{i=1}^N\delta_{\lambda_i},
\end{equation}
where $\lambda_1,\ldots,\lambda_N$ are the eigenvalues of $H$, which are all real-valued due to the symmetry of $H$. The relationship between $G$ and $\sigma$ is given by inspecting the Stieltjes transform $s\defeq S_{\sigma}$ of $\sigma$. In general, the Stieltjes transform $S_{\nu}$ of a probability measure $\nu$ on $(\R,\Bcal)$ is given by the map 
\map{S_{\nu}}{\C_+}{\C_+}{z}{\int_{\R} \frac{1}{x-z} \de\nu(x),} 
so using \eqref{eq:ESD} we obtain
\[
s(z) = S_{\sigma}(z) = \int_{\R} \frac{1}{x-z} \de\sigma(x) = \frac{1}{N}\sum_{i=1}^N \frac{1}{\lambda_i-z} = \frac{1}{N}\tr G(z).
\]
As $N\to\infty$ we want to analyze the weak convergence behavior of $\sigma$ to the \emph{semicircle distribution} $\mu$, which is the probability distribution on $(\R,\Bcal)$ with Lebesgue density 
$x\mapsto f_{\sigma}(x)\defeq (2\pi)^{-1}\sqrt{(4-x^2)_+}$. We denote by $m\defeq S_{\mu}$ the Stieltjes transform of $\mu$. Then we obtain with \autocite[32]{BaiSi}:
\[
\forall\, z\in\C_+: m(z) = \frac{-z+\sqrt{z^2-4}}{2}.
\]

\subsection{Main Results}
We are now ready to state the main results of this paper. 
Notationally, whenever a $z\in\C_+$ is considered, we set
\begin{equation}
\label{eq:etaEkappa}
 \eta\defeq\eta(z)\defeq \Im(z), \quad E\defeq E(z)\defeq \Re(z) \quad \text{and} \quad \kappa\defeq\kappa(z)\defeq \abs{\abs{E}-2}.
 \end{equation}

\begin{theorem}
\label{thm:WLL}
Fix $\tau\in(0,1)$ and define the domains
\[
\Dcal_N(\tau) \defeq \left[-\frac{1}{\tau},\frac{1}{\tau}\right] + i\left[\frac{1}{N^{1-\tau}},\frac{1}{\tau}\right] \qquad \text{and} \qquad \Dcal^*_N(\tau) \defeq \left[-2+\tau,2-\tau\right] + i\left[\frac{1}{N^{1-\tau}},\frac{1}{\tau}\right].
\] 
Let $H$ be a Curie-Weiss type ensemble, $G(z)=(H-z)^{-1}$ and \[
\Lambda(z)\defeq \max_{i,j}\abs{G_{ij}(z)-m(z)\delta_{ij}}.
\] 
Then it holds 
\begin{equation}
\label{eq:WLL}	
\max(\Lambda(z),\abs{s(z)-m(z)}) \prec \frac{\frac{1}{\sqrt{N\eta}}}{\sqrt{\kappa+\eta+\frac{1}{\sqrt{N\eta}}}}, \qquad z\in \Dcal_N(\tau)
\end{equation}
so particular
\begin{equation}
\label{eq:WLLbulk}
 \max(\Lambda(z),\abs{s(z)-m(z)}) \prec \frac{1}{\sqrt{N\eta}}, \qquad z\in \Dcal^*_N(\tau).	
\end{equation}
\end{theorem}
Note that each \eqref{eq:WLL} and \eqref{eq:WLLbulk} are to be viewed as two separate statements in that each of the terms in the maximum is dominated by the error term on the right hand side. By properties of $\prec$, this is equivalent to the maximum being dominated. For corollaries and many implications of Theorem~\ref{thm:WLL}, we refer the reader to Appendix~\ref{sec:appendix}. 

\begin{remark}
In the literature, the statement of the form of Theorem~\ref{thm:WLL} is called \emph{weak local law} -- see Proposition 5.1 in \autocite{AnttiLLSurvey} and Theorem 7.1 in \autocite{Erdos} -- since in the study of independent entries, smaller error bounds are known to hold (except for the term $\Lambda(z)$ in \eqref{eq:WLLbulk}). The authors of the current paper plan to derive such stronger results also for Curie-Weiss type ensembles. It should also be noted that our error term is slightly smaller than those in the cited statements, since
\begin{equation}
\label{eq:smallerError}
\frac{\frac{1}{\sqrt{N\eta}}}{\sqrt{\kappa+\eta+\frac{1}{\sqrt{N\eta}}}} \leq \frac{1}{(N\eta)^{\frac{1}{4}}} \qquad \text{and}\qquad \frac{\frac{1}{\sqrt{N\eta}}}{\sqrt{\kappa+\eta+\frac{1}{\sqrt{N\eta}}}} \leq \frac{1}{\sqrt{N\eta\kappa}} .
\end{equation}
However, the error term we use also appears naturally in the works of \autocite{AnttiLLSurvey} and \autocite{Erdos}, who then chose to simplify it. We found it more convenient to work with the term as given.
\end{remark}

\begin{corollary}
\label{cor:betaatmost1}
Let $\beta\in[0,1]$ and  $(H_N)_N$ be a Curie-Weiss($\beta$) ensemble as in Example~\ref{ex:CWensemble}. Then as argued there, $(H_N)_N$ is an ensemble of Curie-Weiss type. Therefore, the local law as in Theorem~\ref{thm:WLL} holds for the Curie-Weiss($\beta$) ensemble. 
\end{corollary}

Next, we would like to analyze what can be said about the Curie-Weiss($\beta$) ensemble $(H_N)_N=(N^{-1/2}X_N)_N$ if $\beta>1$. Here, $\ell$-point correlations $\E X_1X_2\cdots X_{\ell}$ -- where $X_1\ldots,X_{\ell}$ are distinct random spins in $X_N$ -- do not vanish as $N\to\infty$.  In \autocite{Kirsch:Kriecherbauer:2018} it was shown that the semicirlce law holds in probability for the ensemble $((1-c(\beta)^2)^{-1/2}H_N)_N$, where $c(\beta)>0$ defines the unique solution in $(0,1)$ of the equation $\tanh(\beta c)=c$. Additionally, by the work in \autocite{Fleermann:Kirsch:Kriecherbauer:2021} it immediately follows that the semicircle law holds almost surely for $((1-c(\beta)^2)^{-1/2}H_N)_N$. Now, the question is whether the local law also holds locally for $((1-c(\beta)^2)^{-1/2}H_N)_N$.
\begin{theorem}
\label{thm:betalarger1}
Let $\beta>1$ and $(H_N)_N$ be a Curie-Weiss($\beta$) ensemble as in Example~\ref{ex:CWensemble}. Then for the rescaled ensemble $((1-c(\beta)^2)^{-1/2}H_N)_N$ the local semicircle law holds, that is,
\begin{equation}
\label{eq:betalarger1WLL}	
\abs{s(z)-m(z)} \prec \frac{\frac{1}{\sqrt{N\eta}}}{\sqrt{\kappa+\eta+\frac{1}{\sqrt{N\eta}}}}, \qquad z\in \Dcal_N(\tau)
\end{equation}
as well as 
\begin{equation}
\label{eq:betalarger1WLLbulk}
\abs{s(z)-m(z)} \prec \frac{1}{\sqrt{N\eta}}, \qquad z\in \Dcal^*_N(\tau).	
\end{equation}
\end{theorem}

Theorem~\ref{thm:betalarger1} is proved by showing that a rank-1 perturbation of $(H_N)_N$ is, in fact, of Curie-Weiss type as in Definition~\ref{def:CWtype}, and by seeing that a rank-1 perturbation does not affect the local law for $\abs{s(z)-m(z)}$.

Let us explain the heuristics behind the proof of Theorem~\ref{thm:betalarger1}, using the notation of Theorem~\ref{thm:curiedefinetti}. If $\beta>1$, it is well-known that the mixing distribution $\mu_N^{\beta}$ will converge weakly to the probability measure $1/2\cdot\delta_{-c} + 1/2\cdot\delta_{c}$ where $c=c(\beta)$ is as above. As a result, for large $N$ the entries in the upper right triangle of  $X_N=\sqrt{N}H_N$, where $H_N$ is a Curie-Weiss($\beta$) ensemble, are approximately either i.i.d.\ $P_c$ or $P_{-c}$ distributed (with corresponding means $c$ resp.\ $-c$ and variance $1-c^2$ in both cases), each with probability $1/2$, depending on if $\mu_N^{\beta}$ drew $c$ or $-c$. For each of these cases, we standardize $X_N$ so that its upper right triangle contains i.i.d.\ standardized entries, resembling the Wigner case. To this end, denote by $\Ecal_N$ the $N\times N$ matrix consisting entirely of ones, and set
\[
Y_N \defeq \frac{1}{\sqrt{1-c^2}} (X_N - \one_{S_N>0}c\Ecal_N + \one_{S_N\leq 0}c\Ecal_N),
\]
where $S_N$, the sum of the spins of the upper right triangle of $X_N$, is a proxy to decide whether $\mu_N^{\beta}$ tends to $c$ or $-c$. Now for each realization, $Y_N$ is just a rank 1 perturbation of $(1-c^2)^{-1/2} X_N$, leaving the limiting spectral distribution of the $N^{-1/2}$-normalized ensemble unchanged. This was the initial approach in \autocite{Kirsch:Kriecherbauer:2018}. In our situation it is unclear whether $Y_N$ is of de-Finetti type, so we do not know whether it is a Curie-Weiss type ensemble as in Definition~\ref{def:CWtype}. The solution is to find a different random variable to decide when to add or subtract $c\Ecal_N$. The key idea now is that we use a very specific construction of the probability space on which our Curie-Weiss ensembles are defined. We use the product space $\otimes_{N\in\N}(\Mcal_N\otimes\Scal_N)$, where $\Mcal_N\defeq((-1,1),\Bcal((-1,1)),\mu_{N^2}^{\beta})$ and $\Scal_N$ equals $\{\pm 1\}^{N\times N}$, the latter equipped with kernels $(P^{\otimes N^2}_t)_{t\in(-1,1)}$. On each factor $(\Mcal_N\otimes\Scal_N)$ we define the probability measure $\mu_{N^2}^{\beta}(\de t)\otimes P^{\otimes N^2}_t$, which is the product of a probability measure and a kernel. Denote by $M_N^{\beta}$ resp. $\tilde{X}_N$ the projection onto the first resp.\ second component of  $(\Mcal_N\otimes\Scal_N)$, then we obtain mixing variables $M_N^{\beta}$ which is $\mu_{N^2}^{\beta}$ distributed alongside Curie-Weiss($\beta,N^2$)-distributed random variables $\tilde{X}_N$ which are utilized in Example~\ref{ex:CWensemble} to produce the Curie-Weiss($\beta$) ensemble $X_N$. Now we consider 
\[
Z_N \defeq \frac{1}{\sqrt{1-c^2}} (X_N - \one_{M_N^{\beta} >0}c\Ecal_N + \one_{M_N^{\beta}\leq 0}c\Ecal_N).
\]
The ensemble $Z_N$ is, in fact, of de-Finetti type:
\begin{lemma}
\label{lem:largebetadefinetti}
Set $Z_{ij}\defeq Z_N(i,j)$ for all $(i,j)\in I\defeq \{1\leq i\leq j \leq N\}$. Then $Z_I$ is of de-Finetti type with mixing space $((-1,1),\Bcal_{(-1,1)},\mu_{N^2}^{\beta},\tilde{P})$, where for all $t\in(-1,1)$ and with $c=c(\beta)$ as above, we define
\[
\tilde{P}_t\defeq
\begin{cases}
\frac{1+t}{2}\delta_{\frac{1-c}{\sqrt{1-c^2}}} + \frac{1-t}{2}\delta_{\frac{-1-c}{\sqrt{1-c^2}}}& \quad t>0,\\
\frac{1+t}{2}\delta_{\frac{1+c}{\sqrt{1-c^2}}} + \frac{1-t}{2}\delta_{\frac{-1+c}{\sqrt{1-c^2}}}& \quad t\leq 0.
\end{cases}
\]
Further, denoting by $\tilde{m}^{(\ell)}(t) \defeq \int_{\R} x^{\ell} \de \tilde{P}_t$ the $\ell$-th moment of $\tilde{P}_t$, we obtain
\begin{align*}
\tilde{m}^{(1)}(t) & =\begin{cases}
 \frac{1}{\sqrt{1-c^2}}(t-c), & \quad t>0,\\
 \frac{1}{\sqrt{1-c^2}}(t+c), & \quad t\leq 0,
\end{cases}\\
1 - \tilde{m}^{(2)}(t)&=\begin{cases}
 \frac{2c}{1-c^2}(t-c), & \quad t>0,\\
 \frac{-2c}{1-c^2}(t+c), & \quad t\leq 0.
\end{cases}
\end{align*}  
\end{lemma}
\begin{proof}
For the duration of this proof, set $X_{ij}\defeq X_N(i,j)$ for all $(i,j)\in I$. We observe that $Z_I$ takes values in 
\[
\left\{\frac{\pm 1-c}{\sqrt{1-c^2}}\right\}^I \dot{\cup} \left\{\frac{\pm 1+c}{\sqrt{1-c^2}}\right\}^I.
\]
Now let $z_I$ be an arbitrary element in above set, w.l.o.g.\ $z_I\in \{(\pm 1+c)/\sqrt{1-c^2}\}^I$. Let $y_I\defeq \sqrt{1-c^2}z_I$ and $x_I\defeq y_I-c\in\{\pm 1\}^I$. Then
\begin{align*}
\Prob(Z_I = z_I) &= \Prob(X_I - c\one_{M_N^{\beta}>0} + c\one_{M_N^{\beta} \leq 0} = y_I)\\
&= \int_{(-1,1)}\Prob(X_I - c\one_{M_N^{\beta}>0} + c\one_{M_N^{\beta} \leq 0} = y_I | M_N^{\beta}=t)\Prob^{M_N^{\beta}}(\de t)\\
&= \int_{(-1,0)}\Prob(X_I = x_I | M_N^{\beta}=t)\Prob^{M_N^{\beta}}(\de t) \ = \ \int_{(-1,1)}\tilde{P}_t^{\otimes I}(z_I)\mu_{N^2}^{\beta}(\de t),
\end{align*}
The moment calculations for $\tilde{m}^{(1)}$ and $1-\tilde{m}^{(2)}$ are straightforward.
\end{proof}

\begin{lemma}
\label{lem:largebetaCWtype}
The ensemble $N^{-1/2}Z_N$ -- which is a rank one perturbation of $((1-c(\beta)^2)^{-1/2}H_N)_N$ -- is a Curie-Weiss type ensemble as in Definition~\ref{def:CWtype}.	
\end{lemma}
\begin{proof}
In Lemma~\ref{lem:largebetadefinetti} we have just shown that $(Z_N(i,j))_{1\leq i\leq j\leq N}$ is of de-Finetti type with mixture $\mu_{N^2}^{\beta}$	as in Theorem~\ref{thm:curiedefinetti}, but with map $t\mapsto \tilde{P}_t$ as in Lemma~\ref{lem:largebetadefinetti}. Thus, condition a) of Definition~\ref{def:CWtype} is satisfied. It remains to verify conditions \eqref{eq:firstmoment}, \eqref{eq:secondmoment}, \eqref{eq:centralfirstmoment} and \eqref{eq:centralsecondmoment}. Note that in our setting, only the mixing distribution $\mu_{N^2}^{\beta}$ depends on $N$, but not the associated space $T=T_N=(-1,1)$, nor the maps  $t\mapsto \tilde{P}_t$ and the moments $\tilde{m}^{(l)}(t)$. For the proof, we need two well-known facts about the distributions $\mu_{N^2}^{\beta}$ on $(-1,1)$ when $\beta>1$, see e.g. Lemma 6 in \autocite{Kirsch:Kriecherbauer:2018} (where $c=c(\beta)\in(0,1)$ such that $\tanh(c\beta)=c$):
\begin{itemize}
\item[(CW1)] $\exists\, C,\delta>0: \forall\, N\in\N: \mu_{N^2}^{\beta}([-c/2,c/2]) \leq C e^{-\delta N^2}$,
\item[(CW2)] $\forall\, \ell\in\N: \exists\, C_{\ell}>0:\forall\,N\in\N: \int_{c/2}^1\abs{t-c}^{\ell}\mu^{\beta}_{N^2}(\de t) \leq \frac{C_{\ell}}{N^{\ell}}$.
\end{itemize}
To show \eqref{eq:firstmoment}, we calculate for $p\in 2\N$:
\begin{align*}
&\int_{(-1,1)}\abs{\tilde{m}^{(1)}(t)}^p\mu_{N^2}^{\beta}(\de t) = 2\int_{(0,1)}\left(\frac{t-c}{\sqrt{1-c^2}}\right)^p\mu_{N^2}^{\beta}(\de t)\\
&\leq 2 \left(\frac{c}{\sqrt{1-c^2}}\right)^p Ce^{-\delta N^2} + \frac{2C_p}{(\sqrt{1-c^2})^pN^p} \leq \frac{Const(p,c)}{N^p} \leq \frac{Const(p,c)}{N^{p/2}},
\end{align*}
where in the first step, we used Lemma~\ref{lem:largebetadefinetti} and that the measure $\mu_{N^2}^{\beta}$ is symmetric. In the second step, we split integration over $(0,c/2)$ and $(c/2,1)$ and used (CW1) and (CW2). This shows \eqref{eq:firstmoment}, and \eqref{eq:secondmoment} can be shown analogously, since again -- by Lemma~\ref{lem:largebetadefinetti} -- we basically integrate over $\abs{t-c}^p$ resp. $\abs{t+c}^p$. Conditions  \eqref{eq:centralfirstmoment} and \eqref{eq:centralsecondmoment} are satisfied since there is a compact subset of $\R$ in which the support of every probability measure $\tilde{P}_t$ is contained. 
\end{proof}
The last ingredient for the proof of Theorem~\ref{thm:betalarger1} is the following lemma:
\begin{lemma}
\label{lem:rankoneperturb}
Let $Y$ be an Hermitian $N\times N$ matrix, $\Ecal$ be an arbitrary Hermitian $N\times N$ matrix of rank $k$. Then it holds for all $z\in\C_+$:
\[
\bigabs{\tr\left[(Y-z)^{-1}\right]-\tr\left[(Y+\Ecal-z)^{-1}\right]} \leq \frac{2k}{\eta}.
\]
\end{lemma}
\begin{proof}
Unitary transformations of $Y$ and $Y+\Ecal$ do not affect the l.h.s.\ of the statement, so we may assume $\Ecal(1,1),\ldots,\Ecal(k,k)\neq 0$ but all other entries of $\Ecal$ vanish. We define for all $i=1,\ldots, k$ the matrix $\Ecal_i$ such that $\Ecal_i(j,j)=\Ecal(j,j)$ for all $j\in\{1,\ldots,i\}$ but all other entries of $\Ecal_i$ vanish. In particular, $\Ecal_k = \Ecal$. Further let $\Ecal_0$ be the matrix consisting entirely of zeros. Then -- denoting for any $N\times N$ matrix $M$ and $l\in\{1,\ldots,N\}$ by $M^{(l)}$ the $l$-th principal minor of $M$ -- we calculate
\begin{align*}
&\tr\left[(Y-z)^{-1}\right]-\tr\left[(Y+\Ecal-z)^{-1}\right]
=\sum_{i=0}^{k-1}\left( \tr\left[(Y+\Ecal_i-z)^{-1}\right] - \tr\left[(Y+\Ecal_{i+1}-z)^{-1}\right]\right)\\
&=\sum_{i=0}^{k-1} \left(\tr\left[(Y+\Ecal_i-z)^{-1}\right] - \tr\left[((Y+\Ecal_{i})^{(i+1)}-z)^{-1}\right]\right)\\ 
&\qquad + \sum_{i=0}^{k-1}\left(\tr\left[((Y+\Ecal_{i+1})^{(i+1)}-z)^{-1}\right] - \tr\left[(Y+\Ecal_{i+1}-z)^{-1}\right]\right)
\end{align*}
where for the second equality we used that $((Y+\Ecal_{i})^{(i+1)} = ((Y+\Ecal_{i+1})^{(i+1)}$ for all $i=0,\ldots,k-1$. Taking absolute values, applying the triangle inequality and then the bound
(A.1.12) in \autocite{BaiSi} yields the statement.
\end{proof}

\begin{proof}[Proof of Theorem~\ref{thm:betalarger1}]
Let $\beta>1$ and $(H_N)_N$ be Curie-Weiss($\beta$) ensemble. Denote by $s_N$ the Stieltjes transform corresponding to $((1-c(\beta)^2)^{-1/2}H_N)_N$ and by $\tilde{s}_N$ the Stieltjes transformation corresponding to the ensemble $N^{-1/2}Z_N$ as defined above, which is of Curie-Weiss type and a rank one perturbation of $((1-c(\beta)^2)^{-1/2}H_N)_N$ by Lemma~\ref{lem:largebetaCWtype}. 
\[
\abs{\tilde{s}_N(z) - s_N(z)} \prec \frac{1}{N\eta} \qquad ,z\in\C_+
\]
by Lemma~\ref{lem:rankoneperturb}. The proof is concluded by using the estimates on $\abs{\tilde{s}_N - m}$ obtained by Theorem~\ref{thm:WLL}.
\end{proof}

\section{Proof of Theorem~\ref{thm:WLL}}
For the proof of Theorem~\ref{thm:WLL}, we follow the strategy used in \autocite{AnttiLLSurvey} to prove their Proposition~5.1. Their proof works for independent entries, and it is a key observation that the ingredients which actually use the independence condition are exactly the so-called "large deviation bounds", stated in Lemma~3.6 in \autocite{AnttiLLSurvey}.

To begin the proof, note that is suffices to show the statements in \eqref{eq:WLL} and \eqref{eq:WLLbulk} for $\Lambda(z)$, since then by averaging, we obtain the $\prec$ bounds for $\abs{s(z)-m(z)}$, hence for the maximum. Now we proceed along the lines of \autocite{AnttiLLSurvey} and reveal the changes we made. We introduce the following notation: We write $\Lambda_*\defeq\max_{i\neq j}\abs{G_{ij}}$. Using the Schur complement formula, we obtain
\begin{equation}
\label{eq:Schur}	
\frac{1}{G_{ii}} = H_{ii} - z - \sum_{k,l}^{(i)} H_{ik}G^{(i)}_{kl}H_{li}.
\end{equation}
Here, if $T\subseteq\{1,\ldots,N\}$ is a subset, the sum $\sum_{k,l}^{(T)}$ denotes the sum over all $k,l\in\{1,\ldots,N\}\backslash T$, and $G^{(T)}(z)$ denotes the resolvent of the matrix $(H_{i,j})_{i,j\notin T}$ at $z$. We decompose the expression \eqref{eq:Schur} as follows:
\[
\frac{1}{G_{ii}} = - z - s + Y_i
\]
where $Y_i\defeq H_{ii} + A_i - Z_i$ with $Z_i \defeq Z^{(1)}_i + Z^{(2)}_i$, where
\[
A_i\defeq \frac{1}{N}\sum_k \frac{G_{ki}G_{ik}}{G_{ii}}, \qquad Z^{(1)}_i \defeq \sum_{k\neq l}^{(i)} H_{ik}G_{kl}^{(i)}H_{li}, \qquad Z^{(2)}_i \defeq \sum_{k}^{(i)}\left(\abs{H_{ik}}^2 - \frac{1}{N}\right)G_{kk}^{(i)}.
\]
As it turns out in the analysis of the local law, the only problematic component of the error term $Y_i$ is $Z_i$: Practically all the work the local law requires is to show the smallness of $Z_i$. In what follows we set
\[
\phi \defeq \one_{\Lambda\leq N^{-\tau/10}}.
\]
The following lemma contains the main $\prec$ estimates needed for the proof of Theorem~\ref{thm:WLL}, cf.\ Lemma 5.4 in \autocite{AnttiLLSurvey}. 
\begin{lemma}
\label{lem:mainestimates}
In the above setting, we obtain
\begin{equation}
\label{eq:mainestimates}	
(\phi + \one_{\eta\geq 1}) \left(\Lambda_* + \abs{A_i} + \abs{Z_i} + \abs{G_{ii} -s}\right) \prec \sqrt{\frac{\Im m + \abs{s-m}}{N\eta}},
\end{equation}
uniformly over all $z\in\Dcal_N(\tau)$ and $i\in\{1,\ldots, N\}$
\end{lemma}
Note that \eqref{eq:mainestimates} consists of eight separate $\prec$ statements.
The proof of Lemma~\ref{lem:mainestimates} can be conducted as in \autocite{AnttiLLSurvey}, but since we deal with correlated entries, we need new so called "large deviation bounds" as in Lemma 3.6 in \autocite{AnttiLLSurvey}. Thus, the main work is to establish these bounds in our situation. We present the two-step approach developed in \autocite{FleermannDiss}. In the first step, our Theorem~\ref{thm:largedev} generalizes Lemmas D.1, D.2 and D.3 in \autocite{AnttiLLSurvey} to independent random variables with a common expectation $t\in\C$ which may differ from zero. Notationally, for the remainder of this paper, sums over "$i\neq j\in\oneto{N}$" are over all $i$ and $j$ in $\{1,\ldots,N\}$ with $i\neq j$.

\begin{theorem}
\label{thm:largedev}
Let $N\in \N$ be arbitrary, $(a_{i,j})_{i,j\in\oneto{N}}$ and $(b_i)_{i\in\oneto{N}}$ be deterministic complex numbers, $(Y_i)_{i\in\oneto{N}}$ and $(Z_i)_{i\in\oneto{N}}$ be complex-valued random variables with common expectation $m^{(1)}\in\C$, so that the whole family $\Wcal\defeq \{Y_i\,|\, i\in\oneto{N}\}\cup\{Z_i\,|\, i\in\oneto{N}\}$ is independent. Further, we assume that for all $p \geq 2$ there exists a $\mu_p\in\R_+$ such that $\norm{W-m^{(1)}}_p \leq \mu_p$ for all $W\in\Wcal$. Then we obtain for all $p\geq 2$:
\begin{enumerate}[i)]
\item $
\bignorm{\sum\limits_{i \in \oneto{N}} b_{i} Y_i}_p \leq  \left(A_p \mu_p 
+ \sqrt{N}\abs{m^{(1)}} \right) 
\sqrt{\sum\limits_{i \in \oneto{N}}\abs{b_{i}}^2}$,

\item $
\bignorm{\sum\limits_{i,j \in \oneto{N}} a_{i,j} Y_i Z_j}_p \leq  \left(A_p^2 \mu^2_p 
+  2A_p \mu_p\sqrt{N}\abs{m^{(1)}}  +
N \abs{m^{(1)}}^2\right) 
\sqrt{\sum\limits_{i,j \in \oneto{N}}\abs{a_{i,j}}^2}$,
\item $
\bignorm{\sum\limits_{i\neq j\in\oneto{N}} a_{i,j} Y_i Y_j}_p \leq  \left(4 A_p^2 \mu^2_p 
+ 2 A_p \mu_p \sqrt{N}\abs{m^{(1)}}   +
N \abs{m^{(1)}}^2\right) 
\sqrt{\sum\limits_{i\neq j\in\oneto{N}}\abs{a_{i,j}}^2}$.
\end{enumerate}
where $A_p\in \R_+$ is a constant which depends only on $p$.

\label{theorem}	
\end{theorem}
\begin{proof}
We show statement $iii)$ first. Surely, $(Y_i-m^{(1)})_i$ are centered and uniformly $\norm{\cdot}_p$--bounded by $\mu_p$ for all $p\geq 2$.
 For $p\geq 2$ we find:
\begin{align*}
&\bignorm{\sum_{i\neq j\in\oneto{N}} a_{i,j} Y_i Y_j}_p \leq\bignorm{\sum_{i\neq j\in\oneto{N}} a_{i,j} (Y_i-m^{(1)})(Y_j-m^{(1)})}_p +
\bignorm{\sum_{i\neq j\in\oneto{N}} a_{i,j}m^{(1)}(Y_j-m^{(1)})}_p \\
&\quad+ \bignorm{\sum_{i\neq j\in\oneto{N}} a_{i,j}m^{(1)}(Y_i-m^{(1)}) }_p +
\bignorm{\sum_{i\neq j\in\oneto{N}} a_{i,j} (m^{(1)})^2}_p =: T_1 + T_2 + T_3 + T_4.
\end{align*}
We will now proceed to analyze the four terms separately. To bound $T_1$, we have by Lemma D.3 in \autocite{AnttiLLSurvey} that 
\[
T_1 \leq 4A^2_p \mu_p^2\sqrt{\sum_{i\neq j\in\oneto{N}}\abs{a_{i,j}}^2}.
\]
For $T_2$ (and analogously for $T_3$) we obtain through Lemma D.1 in \autocite{AnttiLLSurvey} that
\begin{align*}
T_2 &= \abs{m^{(1)}} \bignorm{\sum_{j\in\oneto{N}}\left(\sum_{i\in\oneto{N}\backslash\{j\}}a_{i,j}\right) (Y_j-m^{(1)})}_p\\
&\leq \abs{m^{(1)}}A_p \mu_p\sqrt{\sum_{j \in \oneto{N}}\bigabs{\sum_{i\in\oneto{N}\backslash\{j\}}a_{i,j}}^2} 
\leq \sqrt{N} \abs{m^{(1)}} A_p \mu_p\sqrt{\sum_{i\neq j\in\oneto{N}}\abs{a_{i,j}}^2},
\end{align*}
where we used that the Cauchy-Schwarz inequality. Lastly, we obtain
\begin{align*}
T_4 &= \bigabs{\sum_{i\neq j\in\oneto{N}} a_{i,j} (m^{(1)})^2} = \abs{m^{(1)}}^2  \bigabs{\sum_{i\neq j\in\oneto{N}} a_{i,j}}\\
& \leq \abs{m^{(1)}}^2\sqrt{\sum_{i\neq j\in\oneto{N}}\abs{a_{i,j}}^2} \cdot \sqrt{N^2} = N \abs{m^{(1)}}^2 \sqrt{\sum_{i\neq j\in\oneto{N}}\abs{a_{i,j}}^2}.
\end{align*}
This shows that $iii)$ holds. Now $ii)$ is shown analogously to $iii)$, with the difference that sums over $i$ and $j$ are always over $\oneto{N}$ without further restrictions such as $i\neq j$. In addition, instead of using Lemma D.3 in \autocite{AnttiLLSurvey} to bound $T_1$, we then use Lemma D.2 in \autocite{AnttiLLSurvey} (where constants are smaller, thus we can replace $4 A_p^2 \mu_p^2$ by $A_p^2\mu_p^2$).

To show that $i)$ holds, we calculate for $p\geq 2$:
\begin{align*}
\bignorm{\sum_{i\in\oneto{N}} b_iY_i}_p 
&= \bignorm{\sum_{i\in\oneto{N}} b_i((Y_i-m^{(1)})+m^{(1)})}_p
 \leq \bignorm{\sum_{i\in\oneto{N}} b_i(Y_i-m^{(1)})}_p + \bignorm{\sum_{i\in\oneto{N}} b_i m^{(1)}}_p\\
&\leq A_p \mu_p \sqrt{\sum_{i\in\oneto{N}} \abs{b_i}^2} + \abs{m^{(1)}} \bigabs{\sum_{i\in\oneto{N}} b_i}
 \leq (A_p\mu_p + \abs{m^{(1)}}\sqrt{N})\sqrt{\sum_{i\in\oneto{N}} \abs{b_i}^2},
\end{align*}
where in the third step we used Lemma D.1 in \autocite{AnttiLLSurvey} , and in the fourth step we used the Cauchy-Schwarz inequality.
\end{proof}

We proceed to show the main large deviations result in relation to the stochastic order relation $\prec$.

\begin{theorem}
\label{thm:largedev2}
Let for all $N\in\N$, $Y=Y^{(N)}$ and $W=W^{(N)}$ be $N$-dependent objects  that satisfy the following for all $N\in\N$:
\begin{itemize}
\item $W=W^{(N)}$ is a finite index set.
\item $Y_W=(Y_i)_{i\in W}=(Y^{(N)}_i)_{i\in W^{(N)}}=Y^{(N)}_{W^{(N)}}$  is a tuple of random variables of de-Finetti type with respect to some mixing space $(T_N, \mc{T}_N, \mu_N, P^{(N)})$.
\end{itemize}
Further, denote for all subsets $K\subseteq W$ by $\Fcal_W(\R^K)$ the set of tuples $C=(C_i)_{i\in W}$, where for each $i\in W$, $C_i :\R^K\to \C$ is a complex-valued measurable function. Analogously, define for all subsets $K\subseteq W$ by $\Fcal_{W\times W}(\R^K)$ the set of tuples $C=(C_{i,j})_{i,j\in W}$, where for all $i,j\in W$, $C_{i,j} :\R^K\to \C$ is a complex-valued measurable function. 
Then if the mapping $P^{(N)}:T_N\longrightarrow\mc{M}_1(\R)$ satisfies the first moment condition \eqref{eq:firstmoment} and the central first moment condition \eqref{eq:centralfirstmoment}, we obtain the following large deviation bounds:
\begin{enumerate}[i)]
\item $\sum\limits_{i\in I} B_i[Y_K] Y_i \ \prec\  \sqrt{\sum\limits_{i\in I} \abs{B_i[Y_K]}^2}$, uniformly over all pairwise disjoint subsets $I,K\subseteq W$ with $\# I\leq N$, and $B\in\Fcal_W(\R^K)$.
\item $\sum\limits_{i\in I, j\in J} Y_i A_{i,j}[Y_K] Y_j \ \prec\  \sqrt{\sum\limits_{i\in I,j\in J} \abs{A_{i,j}[Y_K]}^2}$, uniformly over all pairwise disjoint subsets $I,J,K\subseteq W$ with $\#I=\#J\leq N$, and $A\in\Fcal_{W\times W}(\R^K)$.
\item $\sum\limits_{i,j\in I, i\neq j} Y_i A_{i,j}[Y_K] Y_j \ \prec\  \sqrt{\sum\limits_{i,j\in I,i\neq j} \abs{A_{i,j}[Y_K]}^2}$, uniformly over all pairwise disjoint subsets $I,K\subseteq W$  with $\#I\leq N$, and $A\in\Fcal_{W\times W}(\R^K)$.
\end{enumerate}
Further, if the mapping $P^{(N)}:T_N\longrightarrow\mc{M}_1(\R)$ satisfies the second moment condition \eqref{eq:secondmoment} and the central second moment condition \eqref{eq:centralsecondmoment}, the same bounds as in i), ii) and iii) hold after replacing $Y_i$ and $Y_j$ on the l.h.s.\ by $1-Y_i^2$ and $1-Y_j^2$, respectively.
\end{theorem}
\begin{proof}
We prove $iii)$ first: Let $\epsilon,D>0$ be arbitrary and choose $p\in 2\N$ with $p\geq 2$ so large that $p\epsilon>D$. Now, we pick an $N\in\N$, then choose pairwise disjoint subsets $I,K\subseteq W^{(N)}$ with $\# I\leq N$ and $A\in \Fcal_{W\times W}(\R^K)$ arbitrarily. To avoid division by zero, we define the set:
\[
\Acal_{3}\defeq\left\{y_K\in\R^K \,\vert\, \sum\limits_{i,j\in I,i\neq j} \abs{A_{i,j}[y_K]}^2 > 0\right\}.
\]
Then we calculate (explanations below, sums over "$i\neq j$" are over all $i,j\in I$ with $i\neq j$):
\begin{align*}
&\Prob\left(
\bigabs{\sum_{i\neq j} Y_i A_{i,j}[Y_K] Y_j} > N^{\epsilon} \left(\sum_{i\neq j}\abs{A_{i,j}[Y_K]}^2\right)^{\frac{1}{2}}
\right)\\
&=\Prob\left(
\bigabs{\frac{\sum_{i\neq j} Y_i A_{i,j}[Y_K] Y_j}{\left(\sum_{i\neq j}\abs{A_{i,j}[Y_K]}^2\right)^{\frac{1}{2}}}}^p \one_{\Acal_{3}}(Y_K)> N^{p\epsilon}
\right)
\, \leq\,\frac{1}{N^{p\epsilon}} \E{\bigabs{\frac{\sum_{i\neq j} Y_i A_{i,j}[Y_K] Y_j}{\left(\sum_{i\neq j}\abs{A_{i,j}[Y_K]}^2\right)^{\frac{1}{2}}}}^p}\one_{\Acal_{3}}(Y_K)\\
&  =\frac{1}{N^{p\epsilon}} \int_{T^{(N)}}\int_{\R^{K}} \int_{\R^{I}} \bigabs{\frac{\sum_{i\neq j} y_i A_{i,j}[y_K] y_j}{\left(\sum_{i\neq j}\abs{A_{i,j}[y_K]}^2\right)^{\frac{1}{2}}}}^p \text{d} P_t^{\otimes I}(y_I) \one_{\Acal_{3}}(y_K)\text{d} P_t^{\otimes K}(y_K)\text{d}\mu_N(t)\\
& \leq\frac{1}{N^{p\epsilon}} \int_{T^{(N)}}\int_{\R^{K}}  
\left[4 A_p^2 K_p^{2/p} 
+ 2 A_p K_p^{1/p} \sqrt{N}\abs{m_N^{(1)}(t)}   +
N \abs{m_N^{(1)}(t)}^2 \right]^p
\text{d} P_t^{\otimes K}(y_K)\text{d}\mu_N(t)\\
& \leq\frac{1}{N^{p\epsilon}} \int_{T^{(N)}}4^p\left[4^p A_p^{2p} K_p^2 
+ 2^p A_p^p K_p N^{\frac{p}{2}}\abs{m_N^{(1)}(t)}^p   +
N^p \abs{m_N^{(1)}(t)}^{2p} \right]
\text{d}\mu_N(t)\\
&\leq\frac{1}{N^{p\epsilon}}\left(4^{2p}A_p^{2p}K_p^{2} + 8^p A_p^p K_p\cdot K_p + K_{2p}\right) \,\leq\, \frac{1}{N^D} \cdot \text{const}(p(\epsilon,D)),
\end{align*}

where the first step follows from the fact that for the event in the probability to hold not all $A_{i,j}[Y_K]$ may vanish,
in the third step we used Lemma~\ref{lem:expdefinetti}, in the fourth step we used part $iii)$ of Theorem~\ref{thm:largedev} (notice that the $\R$-valued coordinates $(y_i)_{i\in I}$ are independent under $P_t^{\otimes I}$ and have expectation $m_N^{(1)}(t)\in\R$, and also $({\int_{\R} \abs{y_i-m_N^{(1)}(t)}^p \text{d}P_t(y_i)})^{1/p} \leq K_p^{1/p}$ by \eqref{eq:centralfirstmoment}, which makes Theorem~\ref{thm:largedev} applicable. Further, $\#I \leq N$),  in the fifth step we used that for $a,b,c\geq 0$ and $p\in\N$ we have $(a+b+c)^p\leq 4^p(a^p + b^p + c^p)$. In the sixth step, we used \eqref{eq:firstmoment} and the Cauchy-Schwarz inequality. Lastly, note that
$\text{const}(p(\epsilon,D))\defeq 4^{2p}A_p^{2p}K_p^2 + 8^p A_p^p K_p^2 + K_{2p}$
 denotes a constant which depends only on $p$, which in turn depends only on the choices of $\epsilon$ and $D$. In particular, this constant does not depend on the choice of $N\in\N$, the sets $I$ and $K$ or the function tuple $A$.  This shows $iii)$. To show $ii)$, we can proceed analogously to the proof of part $iii)$, using part $ii)$ of Theorem~\ref{thm:largedev} instead of part $iii)$. We will show $i)$ in the setting of the last statement, that is, we will replace $Y_i$ by $1-Y_i^2$: Let $\epsilon,D>0$ be arbitrary and choose $p\in 2\N$ with $p\geq 2$ so large that $p\epsilon>D$. Now, we pick an $N\in\N$, then choose pairwise disjoint subsets $I,K\subseteq W^{(N)}$ with $\# I\leq N$ and $B\in \Fcal_{W}(\R^K)$ arbitrarily. To avoid division by zero, we define the set
\[
\Acal_{1}\defeq\left\{y_K\in\R^K \,:\, \sum\limits_{i\in I} \abs{B_{i}[y_K]}^2 > 0\right\}.
\]
Now we calculate, with step-by-step explanations found below, and all sums over $i$ are for $i\in I$:
 \begin{align*}
&\Prob\left(
\bigabs{\sum_{i} (1-Y_i^2) B_{i}[Y_K]} > N^{\epsilon} \left(\sum_{i}\abs{B_{i}[Y_K]}^2\right)^{\frac{1}{2}}
\right)\\
&=\Prob\left(
\bigabs{\frac{\sum_{i} (1-Y_i^2) B_{i}[Y_K]}{\left(\sum_{i}\abs{B_{i}[Y_K]}^2\right)^{\frac{1}{2}}}}^p \one_{\Acal_{1}}(Y_K)> N^{p\epsilon}
\right)
\, \leq\,\frac{1}{N^{p\epsilon}} \E{\bigabs{\frac{\sum_{i} (1-Y_i^2) B_{i}[Y_K]}{\left(\sum_{i}\abs{B_{i}[Y_K]}^2\right)^{\frac{1}{2}}}}^p}\one_{\Acal_{1}}(Y_K)\\
&  =\frac{1}{N^{p\epsilon}} \int_{T^{(N)}}\int_{\R^{K}} \int_{\R^{I}} \bigabs{\frac{\sum_{i} (1-y_i^2) B_{i}[y_K]}{\left(\sum_{i}\abs{B_{i}[y_K]}^2\right)^{\frac{1}{2}}}}^p \text{d} P_t^{\otimes I}(y_I) \one_{\Acal_{1}}(y_K)\text{d} P_t^{\otimes K}(y_K)\text{d}\mu_N(t)\\
& \leq\frac{1}{N^{p\epsilon}} \int_{T^{(N)}}\int_{\R^{K}}  
\left[A_p K_p^{1/p} + \sqrt{N}\abs{1-m_N^{(2)}(t)} 
\right]^p
\text{d} P_t^{\otimes K}(y_K)\text{d}\mu_N(t)\\
& \leq\frac{1}{N^{p\epsilon}} \int_{T^{(N)}}2^p\left[A^p_p K_p + N^{\frac{p}{2}}\abs{1-m_N^{(2)}(t)}^p \right]
\text{d}\mu_N(t)\\
&\leq\frac{1}{N^{p\epsilon}}\left(2^{p}A_p^{p}K_p +  K_p\right) \,\leq\, \frac{1}{N^D} \cdot \text{const}(p(\epsilon,D)),
\end{align*}
 where the first step follows from the fact that for the event in the probability to hold not all $B_{i}[Y_K]$ may vanish,
in the third step we used Lemma~\ref{lem:expdefinetti}, in the fourth step we used part $i)$ of Theorem~\ref{thm:largedev} (notice that the $\R$-valued coordinates $(1-y_i^2)_{i\in I}$ are independent under $P_t^{\otimes I}$ and have expectation $1-m_N^{(2)}(t)\in\R$, and also for all $t\in T^{(N)}$:
\[
\left({\int_{\R} \abs{1-y_i^2-(1-m_N^{(2)}(t))}^p \text{d}P_t(y_i)}\right)^{1/p} \leq K_p^{1/p}
\]
by \eqref{eq:centralsecondmoment}, which makes Theorem~\ref{thm:largedev} applicable. Further, $\#I \leq N$),  in the fifth step we used that for $a,b\geq 0$ and $p\in\N$ we have $(a+b)^p\leq 2^p(a^p + b^p)$. In the sixth step, we used \eqref{eq:secondmoment} and the Cauchy-Schwarz inequality. Lastly, note that
$\text{const}(p(\epsilon,D))\defeq 2^p A_p^{p}K_p + K_p$
 denotes a constant which depends only on $p$, which in turn depends only on the choices of $\epsilon$ and $D$. In particular, this constant does not depend on the choice of $N\in\N$, the sets $I$ and $K$ or the function tuple $B$.  This shows $i)$.
 \end{proof}

The next corollary verifies all applications of large deviation bounds needed for the main estimates, Lemma~\ref{lem:mainestimates}. In particular, with the next corollary at hand, Lemma~\ref{lem:mainestimates} can be proven as in \autocite{AnttiLLSurvey}.

\begin{corollary}
\label{cor:largedevinaction}
In the above setting, we obtain uniformly over all $z\in\C_+$ and $i\neq j\in\{1,\ldots,N\}$:
\begin{align}
&i) \quad Z^{(1)}_i = \sum_{k\neq l}^{(i)} H_{ik}G_{kl}^{(i)}H_{li}	\prec \left(\frac{1}{N^2}\sum_{k,l}^{(i)} \abs{G^{(i)}_{kl}}^2\right)^{\frac{1}{2}}\\
&ii)\quad Z^{(2)}_i = \sum_{k}^{(i)}\left(\abs{H_{ik}}^2 - \frac{1}{N}\right)G_{kk}^{(i)} \prec \left(\frac{1}{N^2}\sum_k^{(i)}\abs{G^{(i)}_{kk}}^2\right)^{\frac{1}{2}}\\
&iii)\quad \sum_{k, l}^{(ij)} H_{ik}G_{kl}^{(ij)}H_{li}	\prec \left(\frac{1}{N^2}\sum_{k,l}^{(ij)} \abs{G^{(ij)}_{kl}}^2\right)^{\frac{1}{2}}
\end{align}	
\end{corollary}
\begin{proof}
We prove $i)$ first: Note that for all $N\in\N$ and $i\in\oneto{N}$, the $N-1$ entries $\sqrt{N}H_{ik}=X_{ik}$, $k\neq i$, are distinct entries from the family $(X_N(a,b))_{1\leq a\leq b\leq N}$, which is of de-Finetti type with mixture $\mu_N$ satisfying the first moment condition \eqref{eq:firstmoment} and the central first moment condition \eqref{eq:centralfirstmoment}. Further, for any $z\in\C_+$  and $k\neq l
\in\{1,\ldots,N\}\backslash\{i\}$ we have that $(H^{(i)}-z)^{-1}(i,j)$ is a complex function of variables in $(X_N(i,j))_{1\leq i\leq j\leq N}$ disjoint from those in $(X_N(i,j))_{1\leq a\leq b\leq N}$. Therefore, the statement follows with Theorem~\ref{thm:largedev2}. Statement $iii)$ is shown analogously, and for statement $ii)$ as well, using the last statement in Theorem~\ref{thm:largedev2}.
	
\end{proof}
Having established the main estimates, Lemma~\ref{lem:mainestimates}, it is time to move to the remainder of the proof of Theorem~\ref{thm:WLL}. This is achieved by adjusting the proof of Lemma 5.7 in \autocite{AnttiLLSurvey} to our setting. To show 
\begin{equation}
\label{eq:toshow}
\Lambda(z)\prec\frac{\frac{1}{\sqrt{N\eta}}}{\sqrt{\kappa+\eta+\frac{1}{\sqrt{N\eta}}}}, \qquad z\in \Dcal_N(\tau),
\end{equation}
we first establish an initial estimate:
\begin{equation}
\label{eq:initialestimate}
	\Lambda\prec \frac{1}{\sqrt{N}},\qquad z \in \Dcal_N(\tau)\cap\{z\in\C, \eta\geq 1\}.
\end{equation}
which can be conducted as in Lemma 5.6 in \autocite{AnttiLLSurvey}. Then, in a second step, we fix $E\in[-\tau^{-1},\tau^{-1}]$ and set $\eta_k\defeq 1 -kN^{-3}$ for all $k=0,1,\ldots,m(N) \defeq \lfloor N^3 - N^{2+\tau}\rfloor$. Then for all these $k$ we find $\eta_k\geq\eta_{m(N)}\geq 1-(N^{3}-N^{2+\tau})N^{-3}=N^{\tau-1}$.

Setting $z_k \defeq E + i\eta_k$ for all $k\in\{0,1,\ldots,m(N)\}$ we now show that
\begin{equation}
\label{eq:bootstrap}
\Lambda(z_k)\prec \frac{\frac{1}{\sqrt{N\eta_k}}}{\sqrt{\kappa+\eta_k+\frac{1}{\sqrt{N\eta_k}}}}, \qquad k\in\{0,1,\ldots,m(N)\},
\end{equation}
where the constants $C_{\epsilon,D}$ do not depend on $E$. Then, by Lipschitzity of all terms involved, this establishes \eqref{eq:toshow}.

To show \eqref{eq:bootstrap}, pick $\epsilon\in(0,\tau/16)$, $D>0$ and set $\delta_k\defeq (N\eta_k)^{-1/2}$. Further, define the sets
\[
\Xi_k\defeq \left\{\Lambda(z_k)\leq N^{3\epsilon}\frac{\delta_k}{\sqrt{\kappa+\eta_k + \delta_k}}\right\}\quad \text{and}\quad \Omega_k\defeq\left\{\abs{s(z_k)-m(z_k)}\leq N^{\epsilon}\frac{\delta_k}{\sqrt{\kappa+\eta_k + \delta_k}}\right\}
\]
Note that our sets $\Xi_k$ deviate from the exposition in \autocite{AnttiLLSurvey} to accomodate our error term. However, the proof goes through as in \autocite{AnttiLLSurvey}, so we will not carry it out here in detail. Eventually, what we achieve is that independently of our initial choice of $E$, 
\[
\sup_{k\in\{0,1,\ldots,m(N)\}} \Prob(\Xi_k^c) \leq N^3(1+N^3)\frac{C_{\epsilon,D}}{N^D},
\]
which establishes \eqref{eq:bootstrap} and thus finishes the proof, since we may choose $D$ arbitrarily large.


\appendix 
\section{Implications of Theorem~\ref{thm:WLL}}
\label{sec:appendix}

Theorem~\ref{thm:WLL} is a statement about the supremum of certain probabilities. It can be strengthened by taking the supremum inside the probability, which is possible due to the Lipschitz continuity of all quantities involved. This will imply that $\prec$ does not only hold uniformly for $z\in\Dcal_N(\tau)$, but also simultaneously for these $z$ (cf. Definition~\ref{def:stochdom}).

We formulate a general theorem, which is of help when lifting uniform $\prec$-statements to simultaneous ones. To this end, in addition to the domains $\Dcal^*_N(\tau)$ and $\Dcal_N(\tau)$, we define the encompassing domains
\[
\forall\, \tau\in(0,1):\, \forall\, N\in\N:\, \Ccal_N(\tau)\defeq\left[-\frac{1}{\tau},\frac{1}{\tau}\right] + i\left[\frac{1}{N},\frac{1}{\tau}\right].
\]
For any sequence of regions $\Gcal_N\subseteq\Ccal_N(\tau)$ and fixed $L\in\N$, define the subsets
\[
\Gcal_N^L\defeq \Gcal_N\cap \frac{1}{N^L}(\Z + i\Z). \label{sym:integers}
\]
For example, we might consider the regions $\Gcal^4_N$ for  $\Gcal_N\defeq \Dcal_N(\tau)$. We notice that $\Gcal_N^4$ forms a $\frac{2}{N^4}$-net in $\Gcal_N$, which means that any $z\in\Gcal_N$ is $\frac{2}{N^4}$-close to some $z'\in\Gcal_N^4$. The following theorem generalizes Remark 2.7 in \autocite{AnttiLLSurvey}.

\begin{lemma}
\label{lem:supinside}	
Suppose we are given stochastic domination of the form
\[
F^{(N)}_i(z) \prec \Psi^{(N)}(z), \qquad i\in I_N,z\in \Gcal^L_N,
\]
where for all $N\in\N$:
\begin{itemize}
\item $\Gcal_N\subseteq \Ccal_N(\tau)$ is a non-empty subset with a geometry such that $\Gcal^L_N$ forms a $\frac{2}{N^L}$-net in $\Gcal_N$.
\item $(F^{(N)}_i)_{i\in I_N}$ is a family of complex-valued functions on $\Ccal_N(\tau)$, where $\# I_N\leq C_1 N^{d_1}$ and for all $i\in I_N$,  $F^{(N)}_i$ is $C_2 N^{d_2}$-Lipschitz-continuous on $\Ccal_N(\tau)$,
\item $\Psi^{(N)}$ is an $\R_{+}$-valued function on $\Ccal_N(\tau)$, which is $C_3 N^{d_3}$-Lipschitz-continuous and bounded from below by $\frac{1}{C_4 N^{d_4}}$,
\end{itemize}
where $C_1,\ldots,C_4 >0$, $d_1,\ldots,d_4>0$ are $N$-independent constants and $L > \max(d_2 + d_4,\ d_3 + d_4)$.
Then we obtain the simultaneous statement:
\begin{equation}
\label{eq:supinside}
\sup_{z\in \Gcal_N} \max_{i\in I_N} \frac{\abs{F^{(N)}_i(z)}}{\Psi^{(N)}(z)} \prec 1.
\end{equation}
\end{lemma}
\begin{proof}
The following statements hold trivially for all $N\in\N$:
\[
i)~ \#\Gcal^L_N \leq\#\Dcal_N^L\leq \frac{3}{\tau} N^L\cdot \frac{2}{\tau} N^L =: C_5 N^{2L},\qquad
ii)~ \forall\,z\in\Gcal_N: \exists\, z'\in\Gcal_N^L:\abs{z-z'}\leq\frac{2}{N^L}.
\]
\underline{Step 1: \eqref{eq:supinside} holds if $\Gcal_N$ is replaced by $\Gcal_N^L$.}\newline
This is easily done by the following calculation for $\epsilon,D>0$ arbitrary:
\[
\Prob\left(\sup_{z\in \Gcal_N^L} \max_{i\in I_N} \frac{\abs{F^{(N)}_i(z)}}{\Psi^{(N)}(z)} > N^{\epsilon}\right)  \leq \sum_{z\in \Gcal_N^L} \sum_{i\in I_N} \Prob\left(\frac{\abs{F^{(N)}_i(z)}}{\Psi^{(N)}(z)} > N^{\epsilon}\right) \leq C_5 N^{2L} C_1 N^{d_1}\frac{C_{\epsilon,D}}{N^D}
\]
This concludes the first step by shifting $D\leadsto D + 2L + d_1$ and absorbing $C_1\cdot C_5$ into $C_{\epsilon,D+2L+d_1}$.
\newline
\underline{Step 2: Extension from $\Gcal_N^L$ to $\Gcal_N$.}\newline
Now, Lipschitz-continuity comes into play: For an arbitrary $\epsilon>0$, suppose 
\[
\exists\, z\in\Gcal_N,\,\exists\,i\in I_N: \abs{F^{(N)}_i(z)} > \Psi^{(N)}(z) N^{\epsilon}.
\]
Then there exists a $z'\in\Gcal_N^L$ with $\abs{z-z'}\leq\frac{2}{N^L}$, and then due to Lipschitz-continuity of $F^{(N)}_i$ and $\Psi^{(N)}$:
\[
\abs{F^{(N)}_i(z')} > \Psi^{(N)}(z') N^{\epsilon} -\frac{2}{N^L} \cdot C_2 N^{d_2} - \frac{2}{N^L}\cdot C_3 N^{d_3+\epsilon} .
\]
It follows, using the lower bound on $\Psi^{(N)}$:
\[
\frac{\abs{F^{(N)}_i(z')}}{\Psi^{(N)}(z')} > N^{\epsilon} - 2\frac{C_2 N^{d_2} + C_3 N^{d_3+\epsilon}}{N^L\Psi^{(N)}(z')} \geq N^{\epsilon} - 2C_4 N^{d_4}\frac{C_2 N^{d_2} + C_3 N^{d_3+\epsilon}}{N^L}.
\]
We may assume w.l.o.g.\ that $\epsilon\in(0,L-d_3-d_4)$ (see Remark~\ref{rem:suffices}). Then 
\[
 \exists\, N(\epsilon)\in\N:\,\forall\,N\geq N(\epsilon): N^{\epsilon} - 2 C_4 N^{d_4}\frac{C_2 N^{d_2} + C_3 N^{d_3+\epsilon}}{N^L} > N^{\frac{\epsilon}{2}}.
\]
 We have shown that for all $N\geq N(\epsilon)$:
 \[
 \left[\exists\, z\in\Gcal_N,\,\exists\, i\in I_N: \frac{\abs{F^{(N)}_i(z)}}{\Psi^{(N)}(z)} > N^{\epsilon}\right]
 \Rightarrow \left[\exists\, z'\in\Gcal^L_N,\,\exists\, i\in I_N: \frac{\abs{F^{(N)}_i(z')}}{\Psi^{(N)}(z')} > N^{\frac{\epsilon}{2}}\right].
 \]
 Therefore, if $D>0$ is arbitrary, we obtain for all $N\geq N(\epsilon)$:
\[
\Prob\left(\sup_{z\in \Gcal_N} \max_{i\in I_N} \frac{\abs{F^{(N)}_i(z)}}{\Psi^{(N)}(z)} > N^{\epsilon}\right)
\leq \Prob\left(\sup_{z\in \Gcal_N^L} \max_{i\in I_N} \frac{\abs{F^{(N)}_i(z)}}{\Psi^{(N)}(z)} > N^{\frac{\epsilon}{2}}\right)\leq \frac{C_{\frac{\epsilon}{2},D}}{N^D},
\] 
where we used Step 1 for the last inequality. This concludes the proof by choosing constants as $(\epsilon,D)\mapsto C_{\frac{\epsilon}{2},D}$ and with Remark~\ref{rem:suffices}.
\end{proof}

We will now show that Theorem~\ref{thm:WLL} actually holds simultaneously.
\begin{theorem}[Simultaneous Local Law for Curie-Weiss-Type Ensembles]
\label{thm:simulWLL}
In the setting of the local law for Curie-Weiss type ensembles (Theorem~\ref{thm:WLL}) we obtain
\begin{equation}
\label{eq:simulWLL}	
\sup_{z\in \Dcal_N(\tau)
}\frac{\max(\Lambda(z),\abs{s(z)-m(z)})}{\frac{\frac{1}{\sqrt{N\eta}}}{\sqrt{\kappa+\eta+\frac{1}{\sqrt{N\eta}}}}} \prec 1
\end{equation}
as well as 
\begin{equation}
\label{eq:simulWLLbulk}
\sup_{z\in \Dcal^*_N(\tau)} \frac{\max(\Lambda(z),\abs{s(z)-m(z)})}{\frac{1}{\sqrt{N\eta}}} \prec 1.
\end{equation}

\end{theorem}
\begin{proof}
Elementary calculations show that on the encompassing domains $\Ccal_N(\tau)$, $\abs{s(z)-m(z)}$ is $2N^2$-Lipschitz and $\Lambda(z)$ is $N^2$-Lipschitz, hence $F^{(N)}(z)\defeq\max(\Lambda(z),\abs{s_N(z)-s(z)})$ is $2N^2$-Lipschitz. Further, on $\Ccal_N(\tau)$ the error terms 
\[
\Psi_1^{(N)}(z)\defeq\frac{\frac{1}{\sqrt{N\eta}}}{\sqrt{\kappa+\eta+\frac{1}{\sqrt{N\eta}}}} \qquad \text{and} \qquad\Psi_2^{(N)}(z)\defeq \frac{1}{\sqrt{N\eta}}
\]
are $3N/\tau$ resp.\ $N/2$-Lipschitz and lower bounded by $\tau/(2\sqrt{N})$ resp.\ $\sqrt{\tau/N}$.
Further, by Theorem~\ref{thm:WLL} we know that 
$F^{(N)}(z) \prec \Psi^{(N)}(z),\ z\in \Dcal^{4}_N(\tau)$.
Therefore, the statement follows directly with Lemma~\ref{lem:supinside}.
\end{proof}

\begin{corollary}
\label{cor:uniformstieltjesconvergence}
In the situation of Theorem~\ref{thm:WLL}, we find
\[
\sup_{z\in \Dcal_N(\tau)}\abs{s(z)-m(z)} \prec \frac{1}{N^{\frac{\tau}{4}}} \qquad \text{and} \qquad \sup_{z\in \Dcal^*_N(\tau)}\abs{s(z)-m(z)} \prec \frac{1}{N^{\frac{\tau}{2}}}
\]	
\end{corollary}
 \begin{proof}
 Since for any $z\in\Dcal_N(\tau)$ we find $
1/(N\eta)^{\frac{1}{4}}\leq 1/\left(N/N^{1-\tau}\right)^{\frac{1}{4}} = 1/N^{\frac{\tau}{4}}$, it follows
\[
\sup_{z\in \Dcal_N(\tau)}\frac{\abs{s(z)-m(z)}}{\frac{1}{N^{\frac{\tau}{4}}}}
\leq \sup_{z\in \Dcal_N(\tau)}\frac{\abs{s(z)-m(z)}}{\frac{\frac{1}{\sqrt{N\eta}}}{\sqrt{\kappa+\eta+\frac{1}{\sqrt{N\eta}}}}}\prec 1
\]
by Theorem~\ref{thm:simulWLL}. Multiplying both sides by $1/N^{\tau/4}$ concludes the proof for the first statement, and the second statement follows analogously.
 \end{proof}
 
 Theorem~\ref{thm:simulWLL} immediately yields Corollary~\ref{cor:uniformstieltjesconvergence}, which allows us to conclude that with high probability, $s$ converges uniformly to $m$ on a growing domain $\Dcal_N(\tau)$ that approaches the real axis. Before venturing further into further corollaries, we recall how Stieltjes transforms can be used to analyze weak convergence, and why it is important for the imaginary part to reach the real axis.

For any probability measure $\nu$ on $(\R,\Bcal)$, there is a close relationship between $S_{\nu}$ and $\nu$, which is observed by analyzing the function  (where $\eta>0$ is fixed)
\begin{equation}
\label{eq:StieltjesImaginary}	
\R\ni E\mapsto \frac{1}{\pi}\Im S_{\nu}(E+i\eta) = \int_{\R}  \frac{1}{\pi}\frac{\eta}{(x-E)^2+\eta^2} \nu(\de x) = (P_{\eta}\ast\nu)(E), 
\end{equation} 
where $\ast$ is the convolution and for any $\eta>0$, $P_{\eta}:\R\to\R$ is the Cauchy kernel, that is, $\forall\,x\in\R: P_{\eta}(x) \defeq \frac{1}{\pi}\frac{\eta}{x^2+\eta^2}$, which is the Lebesgue density function of the Cauchy probability distribution with scale parameter $\eta$. Denoting the Lebesgue measure on $(\R,\Bcal)$ by $\lebesgue$, we find $(P_\eta\ast\nu)\lebesgue=(P_\eta\lebesgue)\ast\nu$, that is, the function in \eqref{eq:StieltjesImaginary} is a well-defined $\lebesgue$-density for the convolution $(P_\eta\lebesgue)\ast\nu$. Further, it can be verified that $i)$ $P_{\eta}\lebesgue\searrow \delta_0$ weakly as $\eta\searrow 0$, $ii)$ the convolution is continuous with respect to weak convergence (if $\nu_n\to\nu$ weakly and $\nu'_n\to\nu'$ weakly, then $\nu_n\ast\nu'_n \to\nu\ast\nu'$ weakly) and $iii)$ the Dirac measure $\delta_0$ is \emph{the} neutral element of convolution. We conclude that $(P_{\eta}\ast \nu)\lebesgue\to\delta_0\ast\nu=\nu$ weakly as $\eta\searrow 0$, which proves the following well-known lemma:

\begin{lemma}
\label{lem:invert}
Let $\nu$ be a probability measure on $(\R,\Bcal)$. Then for any interval $I\subseteq\R$ with $\nu(\partial I)$ = 0, we find:
\[
\nu(I) = \lim_{\eta\searrow 0} [(P_{\eta}\ast\nu)\lebesgue](I) = \lim_{\eta\searrow 0} \frac{1}{\pi}\int_{I} \Im S_\nu(E+i\eta)\lebesgue(\de E).
\]
Thus, any finite measure $\nu$ on $(\R,\Bcal)$ is uniquely determined by $S_\nu$.
\end{lemma}

Let $\sigma_N$ be the ESDs of a sequence of Hermitian $N\times N$ matrices $X_N$. Assume that $\sigma_N$ converges weakly almost surely to the semicircle distribution $\sigma$, that is, convergence takes place on a measurable set $A$ with $\Prob(A)=1$. By the discussion preceding Lemma~\ref{lem:invert}, we find on $A$ that the following commutative diagram holds, where all arrows indicate weak convergence:
\begin{center}
\begin{tikzpicture}
  \matrix (m) [matrix of math nodes,row sep=7em,column sep=6em,minimum width=2em]
  {
     (P_{\eta}\ast\sigma_N)\lebesgue & (P_{\eta}\ast\sigma)\lebesgue \\
     \delta_0\ast\sigma_N=\sigma_N & \sigma \\};
  \path[-stealth]
    (m-1-1) edge node [left] {$\eta\searrow 0$} (m-2-1)
            edge node [below] {$N\to\infty$} (m-1-2)
    (m-2-1.east|-m-2-2) edge node [below] {$N\to\infty$}
            (m-2-2)
    (m-1-2) edge node [right] {$\eta\searrow 0$} (m-2-2)
    (m-1-1) edge node [below left] {$\substack{N\to\infty\\ \eta\searrow 0}$} (m-2-2);
\end{tikzpicture}
\end{center}
\noindent
In particular, the diagonal arrow indicates weak convergence  $(P_{\eta_N}\ast\sigma_N)\lebesgue\to\sigma$ as $N\to\infty$ for any sequence $\eta_N\searrow 0$. But this does not tell us if also densities align, that is, if also $P_{\eta}\ast\sigma_N \to f_{\sigma}$ in some sense, for example uniformly over a specified compact interval. If $\eta=\eta_N$ drops too quickly to zero as $N\to\infty$, then $P_{\eta_N}\ast\sigma_N$ will have steep peaks at each eigenvalue, thus will not approximate the density of the semicircle distribution uniformly. To illustrate this effect, we simulate an ESD of a $100\times 100$ random matrix $X_{100}$, where $(\sqrt{100}X_{100}(i,j))_{1\leq i\leq j\leq 100}$ are independent Rademacher distributed random variables. The density estimates at \emph{bandwidths} $\eta_1\defeq N^{-1/2} = 1/10$ and $\eta_2\defeq N^{-1}=1/100$ are shown in Figure~\ref{fig:goodandbadeta}.

\begin{figure}[htbp]
    \centering
    \begin{minipage}[t]{0.5\linewidth}
        \centering
        \includegraphics[clip, trim=0cm 1.5cm 0cm 1.5cm, width=\linewidth, height=3.9cm]{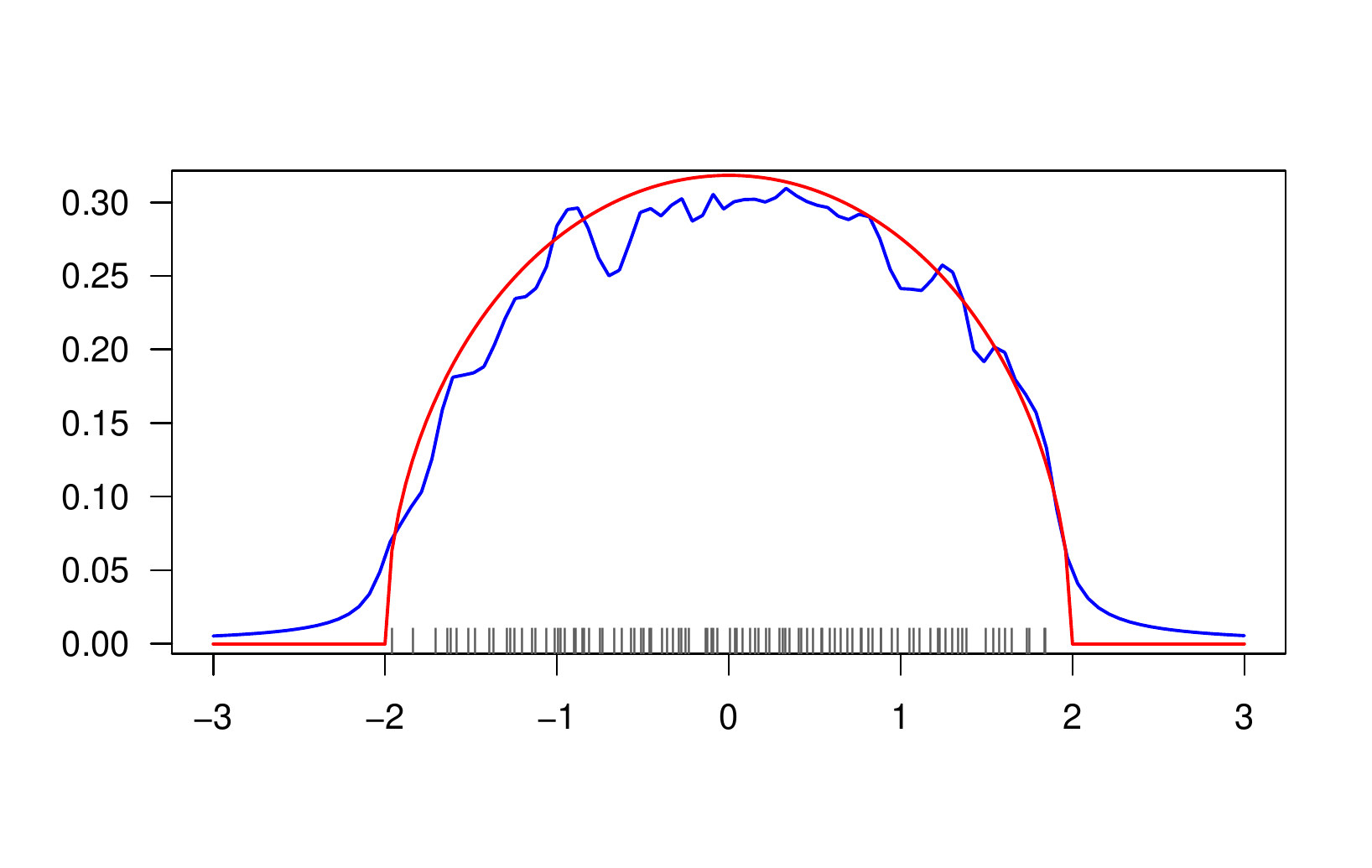}
    \end{minipage}
    \hfill
    \begin{minipage}[t]{0.5\linewidth}
        \centering
        \includegraphics[clip, trim=0cm 1.5cm 0cm 1.5cm, width=\linewidth]{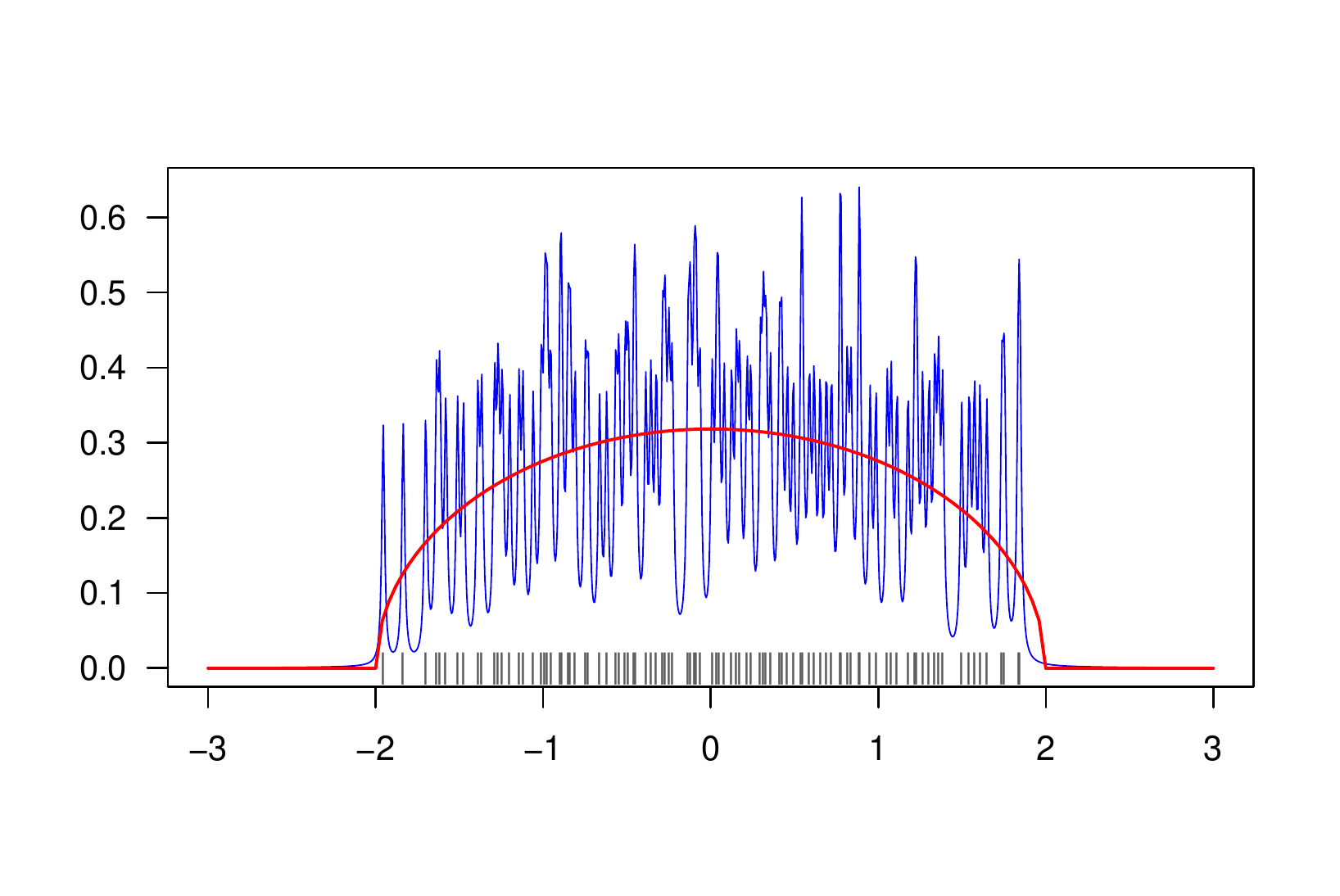}
    \end{minipage}
    \caption{Red lines: $f_{\sigma}$. Blue lines: $\frac{1}{\pi}\Im S_{\sigma_{100}}(\cdot+i\eta)=P_{\eta} \ast \sigma_{100}$. Grey bars: eigenvalue locations. Left figure: $\eta=\eta_1$. Right figure: $\eta=\eta_2$.}
    \label{fig:goodandbadeta}
\end{figure}

As we see in Figure~\ref{fig:goodandbadeta}, we already obtain a decent approximation by the semicircle density when $\eta=\eta_1$, despite the low $N=100$. But after reducing the scale from $\eta_1$ to $\eta_2$, we observe that we do not obtain a useful approximation by the semicircle density anymore. Indeed, the scale $N^{-1}$ is too fast to obtain uniform convergence of the estimated density to the target density, whereas a scale of $N^{-(1-\tau)}$ for any $\tau\in(0,1)$ is sufficient, see our Theorem~\ref{thm:uniformkernelconvergence}, which explains Figure 1 in that it shows that we do have uniform convergence of the densities.

Before we turn to Theorem~\ref{thm:uniformkernelconvergence}, we establish that as $\eta\searrow 0$, the function $E\mapsto\frac{1}{\pi}\Im m(E+i\eta)$, that is $P_{\eta}\ast \sigma$, converges uniformly to $f_\sigma$ over any compact interval and with a speed of $O(\sqrt{\eta})$.
\begin{lemma}
\label{lem:stieltjestosemicircle}
Let $C\geq 2$ be arbitrary, then we obtain for any $\eta\in(0,C]$:
\[
\sup_{E\in[-C,C]} \bigabs{\frac{1}{\pi}\Im(m(E+i\eta)) - f_{\sigma}(E)} \leq  \sqrt{C\eta}.
\]
\end{lemma}
\begin{proof}
Elementary calculations show that if $(a+ib)^2 = c +id$, where $a,c,d\in\R$ and $b>0$, then
\begin{equation}
\label{eq:stieltjestosemicirclehelp}
b = \sqrt{\frac{-c+\sqrt{c^2+d^2}}{2}}.
\end{equation}
With $C\geq 2$ and $z=E+i\eta$, where $E\in[-C,C]$ and $\eta>0$, we find that
$z^2-4 = E^2-\eta^2 - 4 + i 2E\eta$, hence with \eqref{eq:stieltjestosemicirclehelp}:
\[
\frac{1}{\pi}\Im m(z) = -\frac{\Im(z)}{2\pi} + \frac{\Im(\sqrt{z^2-4})}{2\pi} = \frac{1}{2\pi}\left(-\eta+\sqrt{\frac{4+\eta^2 -E^2+\sqrt{(E^2-\eta^2-4)^2 + 4E^2\eta^2}}{2}}\right).
\]	
Assuming at first that $E\in[-2,2]$, we find
\[
\bigabs{\frac{1}{\pi}\Im m(z) - f_{\sigma}(E)}
\leq \frac{\eta}{2\pi} + \frac{1}{2\pi} \left(\sqrt{\frac{4+\eta^2 -E^2+\sqrt{(E^2-\eta^2-4)^2 + 4E^2\eta^2}}{2}} - \sqrt{4-E^2}\right).
\]
Using that $\sqrt{\cdot}$ is uniformly continuous with modulus of continuity $\sqrt{\cdot}$, it suffices to analyze the difference of the arguments, which will then yield the desired upper bound. Now assuming that $E\in[-C,C]\backslash[-2,2]$ we find
\[
\bigabs{\frac{1}{\pi}\Im m(z) - f_{\sigma}(E)}
= \bigabs{\frac{1}{2\pi}\left(-\eta+\sqrt{\frac{4+\eta^2 -E^2+\sqrt{(E^2-\eta^2-4)^2 + 4E^2\eta^2}}{2}}\right)}.
\]
Considering the cases $\eta^2\leq E^2-4$ and $\eta^2> E^2-4$ separately and using $\abs{E}>2$, we obtain
\[
4+\eta^2 -E^2+\sqrt{(E^2-\eta^2-4)^2 + 4E^2\eta^2} \leq  2\eta^2 + 2C\eta,
\] 
which yields the desired upper bound.
\end{proof}

\begin{theorem}
\label{thm:uniformkernelconvergence}	
In the situation of Theorem~\ref{thm:WLL}, define the scale $\eta_N\defeq 1/N^{1-\tau}$ for all $N\in\N$ and assume $\tau < 2/3$. Then 
\[
\sup_{E\in[-\tau^{-1},\tau^{-1}]}\bigabs{\frac{1}{\pi}\Im(s(E+i\eta_N)) - f_{\sigma} (E)} \prec \frac{1}{N^{\frac{\tau}{4}}}.
\]
\end{theorem}
\begin{proof}
Due to Corollary~\ref{cor:uniformstieltjesconvergence},
\[
\sup_{E\in[-\tau^{-1},\tau^{-1}]}\bigabs{\frac{1}{\pi}\Im(s(E+i\eta_N))-\frac{1}{\pi}\Im(m(E+i\eta_N))} \prec \frac{1}{N^{\frac{\tau}{4}}}.
\]
The statement follows with Lemma~\ref{lem:stieltjestosemicircle}, which gives
\[
\sup_{E\in[-\tau^{-1},\tau^{-1}]}\bigabs{\frac{1}{\pi}\Im(m(E+i\eta_N))-f_{\sigma}(E)} \prec \sqrt{\eta_N}=\frac{1}{N^{\frac{1}{2}-\frac{\tau}{2}}}.
\]
\end{proof}

Theorem~\ref{thm:uniformkernelconvergence} states in particular that at the scale $\eta_N = N^{-(1-\tau)}$ ($\tau\in(0,1)$ fixed), we find uniform convergence in probability of $P_{\eta_N}\ast\sigma_N$ to $f_{\sigma}$ on the interval $[-\tau^{-1},\tau^{-1}]$, where we have strong control on the probability estimates.
In his publication \autocite{Khorunzhy1997}, Khorunzhy showed for the Wigner case that for arbitrary but fixed $E\in(-2,2)$ and for slower scales $\eta_N = N^{-(1-\tau)}$ ($\tau\in(3/4,1)$ fixed), $P_{\eta_N}\ast\sigma_N (E)\to  f_{\sigma}(E)$ in probability. Moreover, he showed that this does \emph{not} hold in general for scales that decay too quickly, such as the scale $\eta_N=N^{-1}$, see his Remark $4$ on page 149 in above mentioned publication. See also Figure~\ref{fig:goodandbadeta} on page~\pageref{fig:goodandbadeta} for a visulization of these findings.

We have seen that Theorem~\ref{thm:WLL} and Theorem~\ref{thm:simulWLL} guarantee closeness of the Stieltjes transforms of the ESDs and of the semicircle distribution. Theorem~\ref{thm:uniformkernelconvergence} shows that this implies that $f_\sigma$ can be approximated well by a kernel density estimate $P_{\eta_N}\ast\sigma_N$. 

Next, we state a semicircle law on small scales, which is a probabilistic evaluation of how well the semicircle distribution predicts the fraction of eigenvalues in given intervals $I\subseteq\R$. Interestingly, a variant of the following theorem (see Theorem~\ref{thm:TaoVuWLL} below) even constitutes the local law \emph{per se} in \autocite{Tao:Vu:2012}. Notationally, if $A\subseteq\R$ is a subset, denote by $\Ical(A)$\label{sym:intervalset} the set of all intervals $I\subseteq A$.

\begin{theorem}[Semicircle Law on Small Scales]
\label{thm:smallscales}
In the setting of the local law for Curie-Weiss type ensembles (Theorem~\ref{thm:WLL}), we obtain the two statements
\[
\sup_{I\in\Ical(\R)} \abs{\sigma_N(I)-\sigma(I)} \prec \frac{1}{N^{\frac{1}{4}}} \qquad \text{and}\qquad\sup_{I\in\Ical([-2+\tau,2-\tau])} \abs{\sigma_N(I)-\sigma(I)} \prec \frac{1}{N^{\frac{1}{2}}}.
\]
\end{theorem}
\begin{proof}
The proof can be carried out analogously to the proof of Theorem 2.8 in \autocite{AnttiLLSurvey}.
\end{proof}

Due to Theorem~\ref{thm:smallscales}, for any $\epsilon\in(0,1/4)$ and $D>0$ we find a constant $C_{\epsilon,D}\geq 0$ such that

\begin{equation}
\label{eq:interpretsmallscales1}	
\forall\,N\in\N:~\Prob\left(\sup_{I\in\Ical(\R)}\abs{\sigma_N(I)-\sigma(I)}\leq \frac{N^{\epsilon}}{N^{\frac{1}{4}}} \right) > 1 - \frac{C_{\epsilon,D}}{N^D},
\end{equation}
This tells us that when predicting interval probabilities of $\sigma_N$ by those of $\sigma$, the absolute error will be bounded by $N^{-(1/4-\epsilon)}$.
Note that for small intervals this is not a good statement: Then the error bound of $N^{-(1/4-\epsilon)}$ is useless, since both $\sigma_N(I)$ and $\sigma(I)$ are small anyway. 
The natural way to remedy this would be to consider the relative deviation $\sigma_N(I)/\sigma(I)$. This yields the following theorem, which for Tao and Vu actually constitutes  "The Local Semicircle Law" (instead of a statement as Theorem~\ref{thm:WLL} involving Stieltjes transforms), see their Theorem 7 in \autocite[7]{Tao:Vu:2012}. 

\begin{theorem}[Interval-Type Local Semicircle Laws]
\label{thm:TaoVuWLL}
In the setting of Theorem~\ref{thm:WLL}, we obtain
\begin{enumerate}[i)]
\item For all $\tau\in(0,1/4)$:
\[
\sup_{\substack{I\in\Ical(\R)\\ \abs{I}\geq \frac{1}{N^{1/4-\tau}}}} \frac{\abs{\sigma_N(I)-\sigma(I)}}{\abs{I}} \prec \frac{1}{N^{\tau}}.
\]
\item For all $\tau\in(0,1/2)$:
\[
\sup_{\substack{I\in\Ical([-2+\tau,2-\tau])\\ \abs{I}\geq \frac{1}{N^{1/2-\tau}}}} \bigabs{\frac{\sigma_N(I)}{\sigma(I)}-1} \prec \frac{1}{N^{\tau}}.
\]
\end{enumerate}
\end{theorem}
\begin{proof}
From Theorem~\ref{thm:smallscales} it follows immediately that
\[
\sup_{\substack{I\in\Ical(\R)\\ \abs{I}\geq \frac{1}{N^{1/4-\tau}}}} \frac{\abs{\sigma_N(I)-\sigma(I)}}{\abs{I}} \leq \sup_{I\in\Ical(\R)} \abs{\sigma_N(I)-\sigma(I)}N^{\frac{1}{4}-\tau} \prec \frac{1}{N^{\tau}},
\]
which proves statement i), and ii) can be shown analogously by using the second statement of Theorem~\ref{thm:smallscales}.

\end{proof}

\sloppy
\printbibliography

\vspace{1cm}

\noindent\textsf{(Michael Fleermann and Werner Kirsch)\newline
FernUniversit\"at in Hagen\newline
Fakult\"at f\"ur Mathematik und Informatik\newline
58084 Hagen, Germany}\newline
\textit{E-mail address:}
\texttt{michael.fleermann@fernuni-hagen.de}\\\textit{E-mail address:}
\texttt{werner.kirsch@fernuni-hagen.de}
\vspace{5mm}

\noindent\textsf{(Thomas Kriecherbauer)\newline
Universität Bayreuth\newline
Mathematisches Institut\newline
95440 Bayreuth, Germany}\newline
\textit{E-mail address:}
\texttt{thomas.kriecherbauer@uni-bayreuth.de}

\end{document}